\documentclass[reqno,11pt]{amsart}
\usepackage{epstopdf}
\usepackage{array}
\usepackage{amsthm,amsmath}
\usepackage[authoryear,sort&compress]{natbib}
\usepackage{verbatim}
\usepackage{graphicx}
\usepackage{latexsym}
\usepackage{times}
\usepackage{titlesec} % Very fancy section title control
\titleformat{\section}{\normalfont\bfseries}{\thesection.}{.5em}{}
\titlespacing*{\section}{0pt}{*2}{*1}
\titleformat{\subsection}{\normalfont\bfseries}{\thesubsection.}{.5em}{}
\titlespacing*{\subsection} {0pt}{7pt}{*1}

\usepackage{url}
\usepackage[T1]{fontenc}

\usepackage{nameref}

\usepackage[%
    colorlinks=true,%
    bookmarksnumbered=true,%
    bookmarksopen=true,%
    citecolor=blue,%
    urlcolor=blue,%
    unicode=true,           % enable unicode encoded PDF strings
    breaklinks=true,        % allow links to break over lines by making
    pdftitle={Almost Optimal Sequential Tests of Discrete Composite Hypotheses},%
    pdfauthor={Georgios Fellouris and Alexander G. Tartakovsky},%
]{hyperref}

\usepackage{color}
\newcommand{\mc}[1]{\mathcal{#1}} %% use \mc{symb}
\newcommand{\Fc}{\mc{F}}

\newcommand{\Ic}{\mc{I}}

\newcommand{\mbf}[1]{\mathbf{#1}} %% use {\mb{symb}} provides mathbold roman
\newcommand{\qb}{\mbf{q}}
\newcommand{\pb}{\mbf{p}}

\def\One{\mathchoice{\rm 1\mskip-4.2mu l}{\rm 1\mskip-4.2mu l}%
{\rm 1\mskip-4.6mu l}{\rm 1\mskip-5.2mu l}}
\newcommand\Ind[1]{{\One_{\{#1\}}}}

\newcommand{\set}[1]{\left\{#1\right\}}

\newcommand{\brc}[1]{\left(#1\right)}

\newtheorem{theorem}{Theorem}
\newtheorem{lemma}{Lemma}
\newtheorem{corollary}{Corollary}

\theoremstyle{definition} %% or \theoremstyle{remark} will produce roman text

\newtheorem{remark}{Remark}

\numberwithin{equation}{section} %% double numbering within sections

\newcommand{\cA}{\mathcal{A}}
\newcommand{\cJ}{\mathcal{J}}

\newcommand{\cL}{\mathcal{L}}
\newcommand{\cF}{\mathcal{F}}
\newcommand{\cS}{\mathcal{S}}

\newcommand{\cN}{\mathcal{N}}

\newcommand{\cI}{\mathcal{I}}

\newcommand{\cH}{\mathcal{H}}
\newcommand{\cP}{\mathcal{P}}
\newcommand{\cR}{\mathcal{R}}

\newcommand{\ccab}{\mathcal{C}_{\alpha, \beta}}

\newcommand{\Exp}{{\sf E}}
\newcommand{\Expop}{{\mathbf{E}^{\pi}}}

\newcommand{\Pro}{{\sf P}}
\newcommand{\Prop}{{\mathbf{P}}^{\pi}}
\newcommand{\Hyp}{{\sf H}}

\begin{document}

\renewcommand{\baselinestretch}{1.2}

\hbox{\footnotesize\rm Statistica Sinica (2013): Revised version submitted January 21th, 2013 (Special Siegmund's Issue)}\hfill

\markboth{\hfill{\footnotesize\rm GEORGIOS FELLOURIS AND ALEXANDER TARTAKOVSKY}
\hfill} {\hfill {\footnotesize\rm 	ALMOST OPTIMAL SEQUENTIAL TESTS}
\hfill}

\renewcommand{\thefootnote}{}
$\ $\par

%%%%%%%%%%%%%%%%%%%%%%%%%%%%%%%%%%%%%%%%%%%%%%%%%%%%%%%%%%%%%%%%%

\fontsize{10.95}{14pt plus.8pt minus .6pt}\selectfont
\vspace{0.8pc}
\begin{center}
\large\bf ALMOST OPTIMAL SEQUENTIAL TESTS OF DISCRETE COMPOSITE HYPOTHESES \\
\end{center}
\vspace{2pt}
\centerline{\large\bf }
\vspace{.4cm}
\centerline{Georgios Fellouris and Alexander G.\ Tartakovsky}
\vspace{.4cm}
\centerline{\it The  University of Southern California}
\vspace{.55cm}
\fontsize{9}{11.5pt plus.8pt minus .6pt}\selectfont

%%%%%%%%%%%%%%%%%%%%%%%%%%%%%%%%%%%%%%%%%%%%%%%%%%%%%%%%%%%%%%%%%%%%%

%%Abstract
\begin{quotation}
\noindent {\it Abstract:} We consider the problem of sequentially testing a simple null hypothesis, $\Hyp_{0}$,  versus a composite alternative hypothesis, $\Hyp_{1}$,  that consists of a finite set of densities. We study  sequential tests that are based on thresholding of  mixture-based likelihood ratio statistics and  weighted generalized likelihood ratio statistics. It is shown that both sequential tests have several asymptotic optimality properties as error probabilities go to zero. First, for any weights, they minimize the expected sample size within a constant  term under every scenario in $\Hyp_{1}$ and at least to  first order under $\Hyp_{0}$. 
Second, for appropriate weights that are specified up to a prior distribution, they minimize within an asymptotically negligible term a weighted expected sample size in $\Hyp_{1}$. Third, for a particular prior distribution, they are almost minimax with respect to the expected Kullback--Leibler divergence until stopping. Furthermore,  based on high-order asymptotic expansions for the operating characteristics, we propose prior distributions that lead to a robust behavior. Finally, based on  asymptotic analysis as well as on simulation experiments, we argue that both tests have the same performance when they are designed with the same weights. 

\par

%%Key Words
\vspace{9pt} \noindent {\it Key words and phrases:}  Asymptotic optimality, Generalized likelihood ratio, Minimax sequential tests, Mixture-based tests.
\par
\end{quotation}\par

%%%%%%%%%%%%%%%%%%%%%%%%%%%%%%%%%%%%%%%%%%%%%%%%%%%%%%%%%%%%%%%%%

\fontsize{10.95}{14pt plus.8pt minus .6pt}\selectfont

%%__________________________________
\section{Introduction} \label{s:Intro}

Let $\{X_{t}\}_{t \in \mathbb{N}}$ be a sequence of independent and identically distributed (i.i.d.) random vectors with values in $\mathbb{R}^{d}$, $d \in \mathbb{N}=\{1,2,\dots\}$, and common density $f$ with respect to some non-degenerate, $\sigma$-finite measure $\nu(d x)$.  We consider the problem of \textit{sequentially} testing $\Hyp_{0}: f \in \cA_{0}$ versus $\Hyp_{1}: f \in \cA_{1}$, where $\cA_{0}$ and $\cA_{1}$ are two disjoint sets of densities with common support. That is, 
we assume that observations are acquired in a sequential manner and the goal is to select  the correct hypothesis as soon as possible. 

If  $\{\Fc_{t}\}$ is the observed filtration, i.e.,
$\Fc_{t}= \sigma(X_{1}, \ldots, X_{t})$, a \textit{sequential test}  $\delta=(T, d_T)$ is a pair that consists of an $\{\Fc_t \}$-stopping time, $T$, and an $\Fc_T$-measurable (terminal) decision rule, $d_T=d_T(X_1, \dots,X_T)\in \{0, 1\}$, that specifies which hypothesis is to be accepted once observations have stopped. In particular, $\Hyp_j$ is accepted if $d_{T}=j$, i.e.,  $\set{d_T=j}=\set{T<\infty, ~ \delta ~ \text{accepts $\Hyp_{j}$}}$, $j=0,1$. 

An ideal sequential test should have the smallest possible expected sample size under both $\Hyp_{0}$ and $\Hyp_{1}$, 
while controlling its error probabilities below given tolerance levels. Thus, if $\Pro_{f}$ is the underlying probability measure when $X_{1}$ has density $f$ 
and $\Exp_{f}$ is the corresponding expectation, we will say that $\delta^{o}=(T^o,d^o_{T^o}) \in \ccab$ is an \textit{optimal} sequential test  if 
$$
\Exp_{f}[T^{o}]=\inf_{\delta \in \ccab} \Exp_{f}[T] \quad \forall \; f \in \cA_{0} \cup \cA_{1}, 
$$
where $\ccab$ is the class of sequential tests whose maximal type-I and type-II error probabilities are bounded above by $\alpha$ and $\beta$ respectively, i.e.,
$$
\ccab=\Bigl\{\delta: \sup_{f \in \cA_{0}} \Pro_{f}(d_{T}=1) \leq \alpha \quad \text{and} \quad \sup_{f \in \cA_{1}} \Pro_{f}(d_{T}=0) \leq \beta \Bigr\}.
$$
\citet{WaldWolfowitz48} proved that an optimal sequential test  exists when both hypotheses are simple, i.e., $\cA_{0}=\{f_{0}\}$ and $\cA_{1}=\{f_{1}\}$, and is given by the Sequential Probability Ratio Test (SPRT)  that was proposed by  \citet{Wald44} in his seminal  work on Sequential Analysis:
\begin{equation} \label{sprt}
S= \inf \{t \in \mathbb{N}: \Lambda_{t}^{1} \notin (A^{-1},B)  \} \, , \quad d_{\cS}= \Ind{\Lambda^{1}_{S} \geq B} ,
\end{equation}
where $A,B>1$ are constant thresholds  selected so that $\Pro_{0}(d_{S}=1)=\alpha$ and $\Pro_{1}(d_{S}=0)=\beta$ and 
$\{\Lambda^{1}_{t}\}$ is the likelihood ratio statistic
\begin{equation} \label{lr}
\Lambda^{1}_{t} =  \prod_{n=1}^{t} \frac{f_{1}(X_{n})}{f_{0}(X_{n})} , \quad  t \in \mathbb{N}.
\end{equation}

In the case of composite hypotheses, it has only been possible to find sequential tests that are optimal in an asymptotic sense.
More specifically, we will say that $\delta^{0} \in \ccab$ is  \textit{uniformly (first-order) asymptotically optimal}, if 
\begin{equation*} 
\Exp_{f}[T^{0}]=\inf_{\delta \in \ccab} \Exp_{f}[T] \; (1+o(1)) \quad \forall \; f \in \cA_{0} \cup \cA_{1}, 
\end{equation*}
as $\alpha, \beta \rightarrow 0$. When, in particular, $\cA_{0}$ and $\cA_{1}$ can be embedded in an exponential family $\{f_{\theta}, \theta \in \Theta\}$ and 
$\Theta_{1}$ is a subset of the natural parameter space $\Theta$ so that $\theta_{0} \notin \Theta_{1}$ and 
\begin{equation}  \label{hypotheta}
\cA_{0}= \{f_{\theta_{0}}\} \quad \text{and} \quad \cA_{1}=\{f_{\theta}, \theta \in \Theta_{1}\},
\end{equation}
it is well known  (see, for example,  \cite{Lorden-AS73}, \cite{PollakSiegmund-AS75}) 
that the sequential test (\ref{sprt}) is uniformly asymptotically optimal
if $\Lambda_{t}^{1}$ is replaced either by the generalized likelihood-ratio (GLR) statistic, 
$\sup_{\theta \in \Theta} \Lambda_{t}^{\theta}$, or by a mixture-based likelihood ratio statistic, $\int_{\Theta} \Lambda_{t}^{\theta} \; w(\theta) \, d \theta$, where $w(\cdot)$ is some probability density function on $\Theta$ (weight function) and $\Lambda_{t}^{\theta}$  is defined as in (\ref{lr}) with  $f_{1}$ replaced by $f_{\theta}$. However, apart from certain tractable cases, both these statistics are not in general recursive and, as a result, they cannot be easily implemented on-line. Moreover, their computation at each step may be approximate, since it often requires discretization of the parameter space. These problems can be overcome if one uses the adaptive likelihood-ratio statistic,  $\Lambda_t = \Lambda_{t-1} (f_{\theta_{t}^*}(X_t)/f_0(X_t))$, where $\theta_{t}^{*}$ is an estimator of $\theta$ that depends on the first $t-1$ observations.
However, this approach, initially developed by \citet{RobbinsSiegmund-Berkeley70,RobbinsSiegmund-AS74} for power one tests and later extended by \citet{Pavlov-TPA90} and \citet{DragNov} for multihypothesis sequential tests, generally leads to  less efficient sequential tests, since one-stage delayed estimators use less information than the global MLE that is employed by the GLR statistic. Sequential testing of composite hypotheses in a Bayesian formulation with a small cost of observations was considered by \cite{Schwarz_AMS62, KieferSacks_AMS63, Chernoff72, Lorden_AMS68,Lai_AS88} among others.

In the present paper, we consider the problem of sequential testing  a simple null hypothesis against a \textit{discrete} alternative consisting of a finite set of densities, i.e., we assume that
\begin{equation} \label{hypo3} 
 \cA_{0}=\{f_{0}\} \quad \text{and} \quad \cA_{1}=\{f_{1},\ldots, f_{K}\},
\end{equation}
where $K$ is a positive integer. This hypothesis testing problem has two main motivations.  First, it serves as an approximation to the continuous-parameter testing problem (\ref{hypotheta}), in which $\Theta_{1}$ is replaced by a finite subset $\{\theta_{1}, \ldots, \theta_{K}\}$ of $\Theta_{1}$ so that 
$f_{j}=f_{\theta_{j}}$, $j=0,1,\ldots,K$. Indeed, as we mentioned above, the GLR statistic and mixture-based likelihood ratio statistics cannot always be easily implemented 
on-line and their computation may require discretization of the parameter space. With (\ref{hypo3}), we discretize the alternative hypothesis itself. This implies a loss of efficiency under $\Pro_{\theta}$ when $\theta \notin \{\theta_{1}, \ldots, \theta_{K}\}$, but it leads  to sequential tests that are easily  implementable on-line, a very important advantage for many applications. 

Second, problem (\ref{hypo3}) naturally applies to multisample (also known as \textit{multichannel} or  \textit{multisensor}) slippage problems, which have a wide  range of applications (see, e.g.,~\citet{Chernoff72, Tartakovskyetal-SM06, Tartakovskyetal-IEEEIT03}). As an example, consider the setup in which  $K$ sensors monitor different areas, a signal may be present in at most one of these areas and the goal is to detect  signal presence without identifying its location.  
If additionally the sensors are statistically independent and sensor $i$ takes i.i.d. observations $\{X^{i}_{t}\}_{t \in \mathbb{N}}$ with density $g^{i}_{1}$ (resp. $g^{i}_{0}$) when  signal is present (resp. absent), this problem turns out to be a special case of (\ref{hypo3}) with $X_{t}= (X_{t}^{1}, \ldots, X_{t}^{K})$ and 
\begin{equation}\label{multisensor}
f_{0}(X_{t}) = \prod_{j=1}^{K} g^{j}_{0}(X_{t}^{j}) \; ,  \quad f_{i}(X_{t})= g_{1}^{i}(X_{t}^{i}) \, \prod_{\stackrel{j=1}{ j \neq i}}^{K} g^{j}_{0}(X_{t}^{j}) ,  \quad 1 \leq i \leq K.
\end{equation}

For problem (\ref{hypo3}), we consider two sequential tests which are both parametrized by two vectors with positive components (\textit{weights}), 
$\qb_{j}=(q_{j}^{1}, \ldots, q_{j}^{K})$, $j=0,1$  and they both have the following structure: ``stop the first time $t$ at which 
either $\bar{\Lambda}_{t} \geq B$ or $\underline{\Lambda}_{t}\leq A^{-1}$ and select $\Hyp_{1}$ in the first case and $\Hyp_{0}$ in the latter'',
where  $\{\bar{\Lambda}_{t}\}$ and $\{\underline{\Lambda}_{t}\}$ are appropriate  $\{\cF_{t}\}$-adapted statistics. For the first test, which we call 
{\em Mixture Likelihood Ratio Test} (MiLRT), the corresponding statistics are given by   
$$
\bar{\Lambda}_{t}=\sum_{i=1}^{K} q_{1}^{i} \, \Lambda^{i}_{t} \quad \text{and} \quad \underline{\Lambda}_{t}=  \sum_{i=1}^{K} q_{0}^{i} \, \Lambda^{i}_{t} ;
$$
for the  second test, which we  call {\em Weighted Generalized Likelihood Ratio Test} (WGLRT), they are given by 
$$
\bar{\Lambda}_{t}= \max_{1\leq i \leq K} (q_{1}^{i} \, \Lambda^{i}_{t}) \quad \text{and} \quad \underline{\Lambda}_{t}= \max_{1\leq i \leq K} (q_{0}^{i} \, \Lambda^{i}_{t}),
$$
where $\Lambda_{t}^{i}$ is the  likelihood ratio defined in (\ref{lr}) with $f_{1}$ replaced by $f_{i}$.

\citet{Tartakovskyetal-IEEEIT03} studied the GLRT, i.e., the WGLRT with \textit{uniform} weights, $q_{0}^{i}=q_{1}^{i}=1$, $1 \leq i \leq K$, 
in the multichannel setup (\ref{multisensor}) and established its asymptotic optimality. More  specifically, it was shown that the GLRT 
is \textit{second-order}  asymptotically optimal, in the sense that it attains $\inf_{\delta \in \ccab} \Exp_{i}[T]$ within an $O(1)$ term for every $1 \leq i \leq K$, where $O(1)$ is asymptotically bounded as $\alpha, \beta \to 0$. Moreover, it was shown that, in the special case of completely asymmetric channels, the GLRT also attains $\inf_{\delta \in \ccab} \Exp_{0}[T]$ within an $O(1)$ term. (Here and in what follows we denote by $\Pro_{j}$ the underlying probability measure when $X_{1}$ has density $f_{j}$ and by $\Exp_{j}$ the corresponding expectation, $j=0,1, \ldots, K$.) 

The first contribution of the present work is that this uniform, second-order asymptotic optimality property is established for both the MiLRT and the WGLRT with arbitrary weights $\qb_{0}$ and $\qb_{1}$ in the more general setup of problem (\ref{hypo3}). However, the main question we want to answer is how to select these weights in order to obtain further ``benefits''. In this direction, we show that if $\mathbf{p}=(p_{1}, \ldots, p_{K})$ is an arbitrary probability mass function, which can be interpreted as a prior distribution on $\Hyp_{1}$,  and $\qb_{0}$, $\qb_{1}$ are selected so that
\begin{equation} \label{choose}
q_{0}^{i}= p_{i} \cL_{i} \quad \text{and} \quad q_{1}^{i}= p_{i} / \cL_{i}, \quad 1 \leq i \leq K,
\end{equation} 
then both tests attain $\inf_{\delta \in \ccab}  \Exp^{\pb}[T]$ within an $o(1)$ term, where $\Exp^{\pb}$ is expectation with respect to the weighted probability measure $\Pro^{\pb}=\sum_{i=1}^{k} p_{i} \Pro_{i}$ and  the $\cL$-numbers $\{\cL_{i}\}$, formally introduced in \eqref{L}, provide overshoot corrections that allow us to achieve this refined asymptotic optimality property. 

In addition, we find a prior distribution $\hat{\pb}$ which makes both tests \textit{almost minimax} with respect to the expected Kullback--Leibler (KL) information (divergence) that is accumulated until stopping, in the sense that they attain within an $o(1)$ term
$$\inf_{\delta \in \ccab} \;  \max_{1 \leq i \leq K} \;  (I_{i} \, \Exp_{i}[T]),$$ 
where $I_{i}$ is the KL-information number (see \eqref{I}). In this way, we generalize the corresponding result in \citet{feltar12}, where this minimax problem was considered in the context of open-ended, mixture-based  sequential  tests.

Moreover, we compare numerically the tests with this (almost) least favorable prior distribution  with some alternative choices for $\pb$ in the context of the multichannel problem (\ref{multisensor}) with  channels that take exponential or Gaussian observations. Based on high-order asymptotic expansions for the operating characteristics of both tests, we find that selecting $p_{i}$ to be proportional to $I_{i}$ or $\cL_{i}$ leads to a much more robust behavior than the one  induced by $\hat{\pb}$, especially when the channels have very different signal strengths. Finally, based on these asymptotic expansions as well as on Monte Carlo simulations, we argue that both the WGLRT and the MiLRT have essentially the same performance when they are designed with the same weights.

The remainder of the paper is organized as follows. In Section~\ref{sec:Notation}, we introduce basic notation and present some preliminary results. In Section~\ref{sec:AC}, we obtain asymptotic approximations to the operating characteristics of the two tests, whereas in Section \ref{sec:AM} we establish their asymptotic optimality properties. In Section~\ref{sec:p}, we compare different specifications for $\mathbf{p}$ and 
in Section~\ref{sec:MC} we  compare the tests using Monte Carlo simulations.
We conclude in  Section~\ref{sec:Concl}.

%%_________________________________________________________
\section{Notation, Assumptions and Definitions} \label{sec:Notation}

%%_________________________________________
\subsection{Elements of renewal theory}

For every $1 \leq i \leq K$, we set $Z_{t}^{i}=\log \Lambda_{t}^{i}$, where $\Lambda_{t}^{i}$ is  given by (\ref{lr}) with $f_{1}$ replaced by $f_{i}$. 
We quantify the ``distance'' between $f_{i}$ and $f_{0}$ using the $\cL$-number
\begin{equation} \label{L}
\cL_{i} = \exp\left\{ - \sum_{n=1}^{\infty} n^{-1} \Bigl[\Pro_{0}(Z_{n}^{i}>0)+ \Pro_{i}(Z_{n}^{i} \leq 0) \Bigr] \right\},
\end{equation}
as well as the KL information numbers
\begin{align} 
I_i &= \Exp_i[Z_1^i]= \int \log \Bigl( \frac{f_i(x)}{f_0(x)} \Bigr) f_{1}(x) \, \nu(dx), \label{I}\\
I_0^i &= \Exp_0[-Z_1^i]= \int \log \Bigl(\frac{f_0(x)}{f_i(x)} \Bigr) f_{0}(x) \, \nu(dx).
\end{align}
Without loss  of generality, we assume that $f_{1}, \ldots, f_{K}$ are ordered with respect to their KL divergence from $f_{0}$ so that
\begin{equation} \label{order}
I_{0}=\min_{1\leq i \leq K} I_{0}^{i}= I_{0}^{1} = \cdots = I_{0}^{r} < I_{0}^{r+1} \leq \cdots \leq I_{0}^{K}.
\end{equation} 
Note that  $r=1$ corresponds to the asymmetric situation in which $I_{0}$ is attained by a unique index $i=1$. On the other hand, 
$r=K$ corresponds to the completely symmetric situation in which $I_{0}^{i}$ is the same for every $1\leq i \leq K$. The latter case occurs, for example, 
in the multisample slippage problem \eqref{multisensor} when $g_0^i=g_0$ and $g_1^i=g_1$, $1 \leq i  \leq K$, i.e., when the densities do not depend on the population (or sensor, in a multisensor context).

In order to avoid trivial cases, we assume that $f_{i}$ and $f_{0}$ do not coincide almost everywhere, which implies that $I_{i}$, $I_{0}^{i}>0$ for every $1 \leq i \leq K$. We  also assume throughout the paper that $Z_1^i$ is non-arithmetic under $\Pro_0$ and $\Pro_i$ and that $I_i$, $I_{0}^{i}<\infty$ for every $1 \leq i \leq K$. Then, if we define the first hitting times 
$$
\tau_{c}^{i}= \inf\{t: Z_{t}^{i}\geq c\}, \quad \sigma_{c}^{i}= \inf\{t: Z_{t}^{i} \leq -c\}, \quad c>0,
$$
it is well known that the overshoots $Z_{\tau_{c}^{i}}^{i}- c$ and $|Z_{\sigma_{c}^{i}}^{i}+c|$  have 
well defined asymptotic distributions under $\Pro_{i}$ and $\Pro_{0}$ respectively, i.e.,
$$
\cH_{i}(x) = \lim_{c \rightarrow \infty} \Pro_{i}( Z_{\tau_{c}^{i}}^{i}- c \leq x) , \quad \cH_{0}^{i}(x) =  \lim_{c \rightarrow \infty} \Pro_{0}( |Z_{\sigma_{c}^{i}}^{i} +c|   \leq x ) , \quad x >0,
$$
and consequently, we can define the following Laplace transforms
$$
\gamma_{i}= \int_0^\infty e^{-x} \, \cH_{i}(dx), \quad \gamma_{0}^{i}= \int_0^\infty e^{-x} \, \cH_{0}^{i}(dx),$$
which connect the KL-numbers with the $\cL$-numbers as follows: $\cL_{i} = \gamma_{i} \, I_{i} = \gamma_{0}^{i} \,  I_{0}^{i}$ (see, e.g., Theorem 5 in \cite{Lorden-AS77}). These quantities are very important, since they allow us to achieve with great accuracy the desired error probabilities of the SPRT, $\delta^{i}=(S^{i},d_{S^{i}})$, for testing $f_{0}$ against $f_{i}$ (that is, $\delta^{i}$ is given by  \eqref{sprt} with $\Lambda_t^1$ replaced by $\Lambda_t^i$). Specifically, if  $A=\gamma_{0}^{i}/\beta$ and $B=\gamma_{i}/\alpha$, then $\Pro_{0}(d_{S^{i}}=1)=\alpha (1+o(1))$ and $\Pro_{i}(d_{S^{i}}=0)=\beta(1+o(1))$ as $ \alpha, \beta \rightarrow 0$ (see \cite{sieg-JRSSSB}).

If additionally  second moments are finite, $\Exp_{i}[(Z_{1}^{i})^{2}],  \Exp_{0}[(Z_{1}^{i})^{2}]<\infty$, then $\cH_{i}$ and $\cH_{0}^{i}$ have finite means (average limiting overshoots), 
$$
\kappa_{i}= \int_0^\infty x \, \cH_{i}(dx), \quad \kappa_{0}^{i} = \int_0^\infty x \, \cH_{0}^{i}(dx),  %\quad 1 \leq i \leq K,
$$
and we have the following asymptotic approximations for the expected sample sizes of the SPRT $\delta^{i}$ as $\alpha,\beta \rightarrow 0$ so that $\alpha |\log \beta|+ \beta |\log \alpha| \rightarrow 0$:
\begin{align}
\Exp_{i}[S^{i}] &=\frac{1}{I_i} \brc{|\log \alpha| + \kappa_i +  \log \gamma_{i}} + o(1), \label{ESSi_SPRT} \\
\Exp_{0}[S^{i}] &=\frac{1}{I_{0}^{i}} \brc{|\log \beta| + \kappa_{0}^{i} +  \log \gamma_{0}^{i}} + o(1). \label{ESSi_SPRT0}
\end{align}

%%_________________________________
\subsection{MiLRT and WGLRT}

We will say that $\qb=(q^{1},  \ldots, q^{K})$ is a \textit{weight}, if $q_{i}>0$ $\forall$ $1 \leq i \leq K$. For any weight $\qb$,
we set $|\qb|=\sum_{i=1}^{K} q^{i}$ and we define 
\begin{align}
\Lambda_{t}(\qb) &= \sum_{i=1}^{K} q^{i} \, \Lambda^{i}_{t}, \quad \hat{\Lambda}_{t}(\qb)= \max_{1\leq i \leq K} \{q^{i} \, \Lambda^{i}_{t}\}, \label{LL} \\
Z_{t}(\qb) &= \log \Lambda_{t}(\qb), \quad  \hat{Z}_{t}(\qb) = \log \hat{\Lambda}_{t}(\qb). \label{LZ}
\end{align}
The emphasis of this paper is on the MiLRT, $\delta_{\mathrm{mi}}=(M,d_M)$, and the WGLRT, $\delta_{\mathrm{gl}}=(N,d_N)$, which are  parametrized by two arbitrary weights $\qb_{0}$, $\qb_{1}$ and are defined as follows:
\begin{align*} 
 M &=  \inf \{t :  \Lambda_{t}(\qb_{1}) \geq B  \; \text{or} \; \Lambda_{t}(\qb_{0}) \leq A^{-1}\} , \; d_{M}=\Ind{\Lambda_{M}(\qb_{1}) \geq B} , \\
 N &=  \inf \{t :  \hat{\Lambda}_{t}(\qb_{1}) \geq B  \; \text{or} \; \hat{\Lambda}_{t}(\qb_{0}) \leq  A^{-1} \} , \; d_{N}= \Ind{\hat{\Lambda}_{N}(\qb_{1}) \geq B}.
\end{align*}
Alternatively, if we introduce the following  one-sided stopping times
\begin{align*}
M^{1}_{B}&= \inf \Bigl\{t : \Lambda_{t}(\qb_{1}) \geq  B \Bigr\} , \quad
M^{0}_{A} = \inf \Bigl\{t: \Lambda_{t}(\qb_{0})  \leq A^{-1} \Bigr\} , \\
N^{1}_{B} &= \inf \Bigl\{t :\hat{\Lambda}_{t}(\qb_{1})  \geq B \Bigr\}, \quad
N^{0}_{A} =  \inf\{t: \hat{\Lambda}_{t}(\qb_{0})   \leq  A^{-1} \},
\end{align*} 
$\delta_{\mathrm{mi}}$ and $\delta_{\mathrm{gl}}$ can be defined as follows
\begin{align} 
 M &= \min\{M_{A}^{0}, M_{B}^{1}\}  , \quad  d_{M}= \Ind{M_{B}^{1} \leq M_{A}^{0}}, \label{M} \\
 N &= \min\{N_{A}^{0}, N_{B}^{1} \} , \quad d_{N}= \Ind{N_{B}^{1} \leq N_{A}^{0}}. \label{N}
\end{align}
We also define the associated overshoots
\begin{align} 
\eta &= [Z_{M}(\qb_{1}) - \log B] \, \Ind{d_{M}=1} -[Z_{M}(\qb_{0}) + \log A] \, \Ind{d_{M}=0},  \label{eta}\\
\hat{\eta} &= [\hat{Z}_{N}(\qb_{1}) - \log B] \, \Ind{d_{N}=1} - [\hat{Z}_{N}(\qb_{0}) + \log A] \, \Ind{d_{N}=0},  \label{heta}
\end{align}
which play an important role in the asymptotic analysis of the operating characteristics of the two tests.

%%_______________________________________________________________
\section{Asymptotic Approximations for the  Operating Characteristics}\label{sec:AC}

In this section, we obtain  asymptotic inequalities and approximations for the error probabilities and expected sample sizes of the MiLRT and the WGLRT. In order to do so, we rely on the following decompositions for $Z(\qb)$ and $\hat{Z}(\qb)$, which hold for every $1\leq i \leq K$ and any weight $\qb=(q^{1}, \ldots, q^{K})$,
\begin{align}
Z_{t}(\qb) &=  Z_{t}^{i} + \log q^{i} + Y_{t}^{i}(\qb) , \quad t \in \mathbb{N} ,  \label{Z} \\
\hat{Z}_{t}(\qb) &= Z_{t}^{i} + \log q^{i} + \hat{Y}_{t}^{i}(\qb) , \quad t \in \mathbb{N},  \label{hZ}
\end{align}
where the sequences $Y^{i}(\qb)$ and $\hat{Y}^{i}(\qb)$ are defined as follows:
\begin{align}
Y_{t}^{i}(\qb) &=  \log \brc{1 + \sum_{\stackrel{j=1}{ j \neq i}}^{K} \frac{q^{j}}{q^{i}} \, \frac{\Lambda_{t}^{j}}{\Lambda_{t}^{i}} }, \quad t \in \mathbb{N} \label{Y},
\\
\hat{Y}_{t}^{i}(\qb) &=  \log \brc{ \max \set{1, \max_{1 \leq j \neq i \leq K} \frac{q^{j}}{q^{i}} \, \frac{\Lambda_{t}^{j}}{\Lambda_{t}^{i}}}},
\quad t \in \mathbb{N}. \label{hY} 
\end{align}
From the Strong Law of Large Numbers (SLLN) it  follows that, for every $j \neq i$,
$\Pro_{i}(\Lambda_{t}^{j}/\Lambda^{i}_{t} \rightarrow 0)=1$. This implies that $Y^{i}(\qb)$ and $\hat{Y}^{i}(\qb)$ also converge to 0 $\Pro_{i}$-a.s., and consequently, they are \textit{slowly changing} under $\Pro_{i}$ (for a precise definition of ``slowly changing'' we refer to \citet{sieg-book}, page 190). Since  $Z^{i}_t$ is a random walk under $\Pro_{i}$, from this observation and decompositions (\ref{Z})--(\ref{hZ}) it follows that $Z(\qb)$ and $\hat{Z}(\qb)$ are \textit{perturbed} random walks under $\Pro_{i}$. 

Similarly, the SLLN implies that, \textit{in the special case where $r=1$}, $\Pro_{0}(\Lambda_{t}^{j}/\Lambda^{1}_{t} \rightarrow 0)=1$ for every $j >1$. Therefore, $Y^{1}(\qb)$ and $\hat{Y}^{1}(\qb)$ also converge to 0 $\Pro_{0}$-a.s. and from (\ref{Z})--(\ref{hZ}) with $i=1$ it follows that  $Z(\qb)$ and $\hat{Z}(\qb)$ are perturbed random walks  under $\Pro_{0}$ when $r=1$. 

These properties allow us to apply nonlinear renewal theory for perturbed random walks (see \cite{Woodroofe-AP76, Woodroofe-book82}, \cite{LaiSieg-77, LaiSieg-79}, \cite{sieg-book}) in order to obtain asymptotic approximations for the expected sample sizes of the tests $\delta_{\rm mi}$ and $\delta_{\rm gl}$ under $\Pro_{i}$ for every $1 \leq i \leq K$, as well as under $\Pro_{0}$ when $r=1$. An asymptotic approximation for $\Exp_{0}[N]$ when $r>1$ can be obtained based on the nonlinear renewal theory of  \citet{Zhang-AP88} using the  following representation for $N_{A}^{0}$:
\begin{align} \label{rip}
N_{A}^{0}  &=\inf \Bigl\{t: \ell_{t}^{0} \geq \log A + \max_{1 \leq i \leq K} ( \log q_{0}^{i} + \ell_{t}^{i} )  \Bigr\} ,
\end{align}
where $\ell^{j}$ is the log-likelihood process under $\Pro_{j}$ for $j=0,1, \ldots, K$, i.e.,
\begin{equation} \label{U}
\ell_{t}^{j}=\sum_{n=1}^{t} \log f_{j}(X_{n}), \quad  t \in \mathbb{N}.
\end{equation} 
For the latter approximation we also need some additional notation. Specifically, for any  $1\leq i \leq K$, we set $\mu_{i} = \Exp_{0}[\log f_{i}(X_{1})]$, so that $I_{0}^{i}=\Exp_{0}[\log f_{0}(X_{1})]-\mu_{i}$. Moreover, we set
$\mu= \max_{1 \leq i \leq K} \mu_{i}$, so that $I_{0}= \Exp_{0}[\log f_{0}(X_{1})]-\mu$, we define the  $r$-dimensional random vector 
\begin{equation} \label{W}
W=( \log f_{1}(X_{1})-\mu, \ldots,  \log f_{r}(X_{1})-\mu),
\end{equation}
and we denote by $\Sigma$ its covariance matrix under $\Pro_{0}$. Finally, we set
\begin{equation}\label{drhr}
d_{r}=  \frac{h_{r}}{2 \sqrt{ I_{0}}} , \quad h_{r}= \int_{\mathbb{R}^{r}} (\max_{1 \leq i \leq r} x_{i}) \, \phi_{\Sigma}(x) \, dx ,
\end{equation}
where $\phi_{\Sigma}$ is the density of an $r$-dimensional, zero-mean, Gaussian random vector with covariance matrix $\Sigma$.

%%_____________________________________________
\subsection{Asymptotic bounds for the error probabilities}

We start with the following lemma. 

%%%Lemma
\begin{lemma} \label{L11}
For any $1\leq i\leq K$,
\begin{equation} \label{correct2}
\Exp_{i}[ e^{-\eta} \; \Ind{d_{M}=1}]  \rightarrow \gamma_{i} , \, \quad \Exp_{i}[ e^{-\hat{\eta}} \; \Ind{d_{N}=1}]  \rightarrow \gamma_{i} \quad \text{as}~ A,B \rightarrow \infty.
\end{equation}
If additionally $r=1$, then
\begin{equation} \label{correct}
\Exp_{0}[ e^{-\eta} \; \Ind{d_{M}=0}]  \rightarrow \gamma_{0}^{1} , \, \quad \Exp_{0}[ e^{-\hat{\eta}} \; \Ind{d_{N}=0}]  \rightarrow \gamma_{0}^{1} \quad \text{as}~ A,B \rightarrow \infty.
\end{equation}
\end{lemma}
%%%%%%%%%%%%%%

%%%Proof
\begin{proof}
We will only prove the first assertions in \eqref{correct2} and \eqref{correct}, since the other ones can be proven in an identical way.

Since $M=M_{B}^{1}=\inf \{ t: Z_{t}(\qb_{1})\geq \log B \}$ and $\eta = Z_{M_{B}^{1}}(\qb_{1})- \log B$ on $\{d_{M}=1\}=\{M_{B}^{1} \leq M_{A}^{0}\}$, 
and $\{Z_{t}(\qb_{1}) =  Z_{t}^{i} + \log q_{1}^{i} + Y_{t}^{i}(\qb_{1})\}$ is a perturbed random walk under $\Pro_{i}$, from nonlinear renewal theory (see, e.g.,  Theorem 9.12 in \citet{sieg-book}) it follows that  $\eta$ converges in distribution to $\cH_{i}$ under $\Pro_{i}$ on $\{d_{M}=1\}$. Therefore, the  Bounded Convergence Theorem yields $\Exp_{i}[ e^{-\eta} \; \Ind{d_{M}=1}]  \rightarrow \gamma_{i}$.

Since $M =M_{A}^{0} = \inf \{ t: -Z_{t}(\qb_{0}) \geq \log A \}$ and $\eta = |Z_{M_{A}^{0}}(\qb_{0}) + \log A|$  on $\{d_{M}=0\}=\{M_{B}^{1}>M_{A}^{0}\}$,
and $\{-Z_{t}(\qb_{0}) =  -Z_{t}^{1} - \log q_{0}^{1} - Y_{t}^{1}(\qb_{0})\}$ is a perturbed random walk under $\Pro_{0}$ when $r=1$, 
the same argument as above applies to show that $\Exp_{0}[ e^{-\eta} \; \Ind{d_{M}=0}]  \rightarrow \gamma_{0}^{1}$.
\end{proof}
%%%%%%%%%%%%%%%%%

The following theorem provides exact and asymptotic upper bounds on the error probabilities of $\delta_{\mathrm{mi}}$ and $\delta_{\mathrm{gl}}$.

%%%Lemma
\begin{theorem}  \label{T0}
\begin{enumerate}
\item[(a)]  For any $A,B>1$,
\begin{align} 
& \Pro_{0}(d_{M}=1)  \leq  \frac{|\qb_{1}|}{B} , \quad   \Pro_{0}(d_{N}=1)\leq   \frac{|\qb_{1}|}{B}  \label{pr00},\\
& \Pro_{i}(d_{M}=0) \leq \frac{1}{A \, q_{0}^{i}} \, , \quad \Pro_{i}(d_{N}=0) \leq \frac{1}{A \, q_{0}^{i}} , \quad 1 \leq i \leq K. \label{pr000}
\end{align}

\item[(b)]  As $A,B \rightarrow \infty$,
\begin{align}
\Pro_{0}(d_{M}=1) &=  \frac{1}{B} \, \Bigl(\sum_{j=1}^{K} q_{1}^{j} \, \gamma_{j} \Bigr)  \, (1+o(1)),  \label{pr1} \\
\Pro_{0}(d_{N}=1)  &\leq \frac{1}{B} \, \Bigl(\sum_{j=1}^{K} q_{1}^{j} \, \gamma_{j}\Bigr)  \, (1+o(1)). \label{pr3}
\end{align}
If additionally $r=1$, then for every $1 \leq i \leq K$ 
\begin{equation} \label{correct3}
\Pro_{i}(d_{M}=0) \leq \frac{\gamma_{0}^{1}}{q_{0}^{i} A} (1+o(1)) , \quad \Pro_{i}(d_{N}=0) \leq \frac{\gamma_{0}^{1}}{q_{0}^{i}A} (1+o(1)) .
\end{equation}
\end{enumerate}
\end{theorem}
%%%%%%%%%%%%%%%

%%Proof
\begin{proof}
Let us define the probability measure $\Pro^{\qb_{1}}= \frac{1}{|\qb_{1}|} \, \sum_{i=1}^{K} q_{1}^{i}  \, \Pro_{i}$ 
and denote  by $\Exp^{\qb_{1}}$ expectation with respect to $\Pro^{\qb_{1}}$. Since 
$$
\frac{d \Pro^{\qb_{1}}}{d \Pro_{0}} \Big|_{\cF_{t}}= \frac{1}{|\qb_{1}|} \sum_{i=1}^{K} q_{1}^{i} \Lambda_{t}^{i}= \frac{1}{|\qb_{1}|} e^{Z_{t}(\qb_{1})},
$$
changing the measure $\Pro_{0}\mapsto \Pro^{\qb_{1}}$ we have
\begin{equation} \label{is1}
\Pro_{0}(d_{M}=1) = |\qb_{1}| \; \Exp^{\qb_{1}}[ e^{-Z_{M}(\qb_{1})}   \, \Ind{d_{M}=1}]  
                                      =  \sum_{i=1}^{K} q_{1}^{i} \, \Exp_{i}[ e^{-Z_{M}(\qb_{1})}  \, \Ind{d_{M}=1}].
   \end{equation}
Since $Z_{M}(\qb_{1}) = \log B +\eta$ on $\{d_{M}=1\}$, we obtain
\begin{equation} \label{is11}
\Pro_{0}(d_{M}=1) = \frac{1}{B} \sum_{i=1}^{K} q_{1}^{i} \, \Exp_{i}[ e^{-\eta}  \, \Ind{d_{M}=1}].  
\end{equation}
Since $\eta$ is positive, the first inequality in (\ref{pr00}) immediately follows from \eqref{is11}, whereas (\ref{pr1}) follows from (\ref{correct2}).  A similar argument as the one that led to (\ref{is1}), along with the fact that $Z_{t}(\qb_{1}) \geq \hat{Z}_{t}(\qb_{1})$, yields
\begin{align} \label{is2}
\Pro_{0}(d_{N}=1) &= \sum_{i=1}^{K} q_{1}^{i} \, \Exp_{i}[ e^{-Z_{N}(\qb_{1})}  \, \Ind{d_{N}=1}] \\
&\leq \sum_{i=1}^{K} q_{1}^{i} \, \Exp_{i}[ e^{-\hat{Z}_{N}(\qb_{1})}  \, \Ind{d_{N}=1}] \leq \frac{1}{B} \sum_{i=1}^{K} q_{1}^{i} \, \Exp_{i}[ e^{-\hat{\eta}}  \, \Ind{d_{N}=1}]. \nonumber
\end{align}
The last inequality and the fact that  $\hat{\eta}$ is positive imply the second inequality in (\ref{pr00}), whereas (\ref{pr3}) follows from (\ref{correct2}).

Finally, changing the measure $\Pro_{i} \mapsto  \Pro_{0}$, we obtain  
\begin{equation} \label{is3}
\Pro_{i}(d_{M}=0) =  \Exp_{0}[ e^{Z_{M}^{i}}  \, \Ind{d_{M}=0}].
\end{equation}
Since $Z_{M}^{i}= Z_{M}(\qb_{0})- \log q_{0}^{i}-  Y_{M}^{i}(\qb_{0})$ (recall (\ref{Z})), $Z_{M}(\qb_{0})=-\log A- \eta$ on $\{d_{M}=0\}$ (recall (\ref{eta}))
and $Y_{M}^{i}(\qb_{0}) \geq 0$, it follows that $Z_{M}^{i} \leq -\log (A q_{0}^{i}) - \eta$ on $\{d_{M}=0\}$ and, consequently, (\ref{is3}) becomes
$$
\Pro_{i}(d_{M}=0) \leq  \frac{1}{A q_{0}^{i}} \, \Exp_{0}[e^{-\eta} \, \Ind{d_{M}=0}].
$$
Since $\eta$ is positive, we obtain the first inequality in (\ref{pr000}), whereas from (\ref{correct}) we obtain the first inequality in  (\ref{correct3}). The remaining inequalities in (\ref{pr000}) and (\ref{correct3}) can be shown in a similar way. 
\end{proof}

From Theorem \ref{T0}(a) it is clear that when $A, B$ are selected according to 
\begin{equation} \label{choice}
A_{\beta}(\qb_{0})=\frac{1}{\beta \min_{1 \leq i \leq K} q_{0}^{i}} \, , \quad B_{\alpha}(\qb_{1}) =\frac{|\qb_{1}|}{\alpha},
\end{equation}
then $\delta_{\mathrm{mi}}$, $\delta_{\mathrm{gl}} \in \ccab$. Moreover, 
from Theorem \ref{T0}(b) it follows that we can obtain sharper inequalities if we correct for the overshoots selecting $A,B$ as follows  
\begin{equation} \label{B}
A_{\beta}(\qb_{0})=\frac{\gamma_{0}^{1}}{\beta \min_{1 \leq i \leq K} q_{0}^{i}}, 
\quad B_{\alpha}(\qb_{1})= \frac{\sum_{j=1}^{K} q_{1}^{j} \gamma_{j}}{\alpha}.
\end{equation} 
Indeed, with this selection of the thresholds we have  $\Pro_0(d_M=1) = \alpha(1+o(1))$, 
$\Pro_0(d_N=1) \le \alpha(1+o(1))$ and if additionally $r=1$,
$\max_{1 \leq i  \leq K} \Pro_i(d_M=0) \leq \beta(1+o(1))$  and $\max_{1 \leq i  \leq K} \Pro_i(d_N=0) \leq \beta(1+o(1))$.

%____________________________________________________
\subsection{Asymptotic approximations to expected sample sizes}

In order to obtain asymptotic approximations to the expected sample sizes of the MiLRT and the WGLRT, we will make the following assumptions, 
which will be needed for all the results in the  rest of the paper:

\vspace{0.2cm}

\noindent (A1) $\Exp_{i}[(Z_{1}^{i})^{2}]<\infty \quad \text{and} \quad \Exp_{0}[(Z_{1}^{i})^{2}]<\infty, \quad  1 \leq i \leq K$;

\vspace{0.2cm}

\noindent (A2) $\alpha, \beta \rightarrow 0$ so that $|\log \alpha|/|\log \beta| \rightarrow k$, where $k \in (0, \infty)$;

\vspace{0.2cm}

\noindent (A3) For  $T=M$ or $T=N$, $A$ and $B$ are selected so that as $\alpha,\beta\to 0$
\begin{align}
& k_{0} \, \alpha \, (1+o(1)) \leq \Pro_{0}(d_{T}=1) \leq \alpha \, (1+o(1)),   \label{k0} \\ 
& k_{1} \, \beta \, (1+o(1)) \leq  \max_{1 \leq i \leq K}  \Pro_{i}(d_{T}=0) \leq \beta \, (1+o(1)), \label{k1}  
\end{align}
or equivalently,
\begin{align}
&|\log \alpha| +o(1) \leq |\log \Pro_{0}(d_{T}=1)| \leq |\log \alpha| + |\log k_{0}|+o(1),  \label{lk0} \\
& |\log \beta| +o(1) \leq  |\log \max_{1 \leq i \leq K}  \Pro_{i}(d_{T}=0)| \leq |\log \beta| + |\log k_{1}|+o(1), \label{lk1}
\end{align} 
where $k_{0}$, $k_{1} \in (0,1)$ are fixed constants, not necessarily the same for $\delta_{\rm mi}$ and $\delta_{\rm gl}$.
\vspace{0.2cm}

The second moment conditions (A1) on the log-likelihood ratio $Z_{1}^{i}$ are required even for the asymptotic approximations (\ref{ESSi_SPRT})--(\ref{ESSi_SPRT0}) to the performance of the SPRT for testing $f_{0}$ against $f_{i}$. Assumption (A2) concerns the relative rates with which $\alpha$ and $\beta$ go to 0 and requires that $\alpha$ should not go to 0 exponentially faster than $\beta$ and vice-versa. Note, however, that
$\alpha$ can still be much smaller than $\beta$ (or vice versa), a natural requirement in many  applications. 
Assumption (A3) requires that the thresholds  for both the MiLRT and the WGLRT are designed so that the probabilities of the type-I and type-II errors are \textit{asymptotically} bounded by (and at the same time not  much smaller than)  $\alpha$ and $\beta$ respectively. As the following lemma suggests, 
(A3) connects the thresholds $A$ and $B$ with the desired error probabilities $\alpha$ and $\beta$, so that we do not need to impose additional (to (A2)) constraints to the relative rates with which $A$ and $B$ go to infinity.

%%Lemma
\begin{lemma}  \label{newlemma}
If {\rm (A3)}  holds, then  $\log B= |\log \alpha|+ O(1)$ and $\log A = |\log \beta|+O(1)$.
\end{lemma}

%%Proof
\begin{proof}
From  (\ref{pr00}) we know that $\log B \leq |\log \Pro_{0}(d_{M}=1)|+ |\qb_{1}|$, whereas from (A3), and in particular (\ref{lk0}),
it follows that  $|\log \Pro_{0}(d_{M}=1)| \leq |\log \alpha| + |\log k_{0}| +o(1)$, which proves $\log B = |\log \alpha|+O(1)$. 
The second relationship can be shown in a similar way.
\end{proof}
%%

%%%Theorem
\begin{theorem}  \label{T1}
If conditions {\rm (A1)--(A3)} hold, then 

\noindent {\rm (a)} for every $1\leq i \leq K$, 
\begin{align}  
I_{i} \, \Exp_{i}[M] &= \log B+ \kappa_{i} - \log q^{i}_{1} +o(1),  \label{highM}
\\
I_{i} \, \Exp_{i}[N] &=  \log B+ \kappa_{i} - \log q^{i}_{1} +o(1); \label{highN}
\end{align}

\noindent {\rm (b)} for $r=1$,
\begin{align} 
I_{0} \, \Exp_{0}[M] &= \log A  + \kappa_{0}^{1} + \log q_{0}^{1} +o(1), \label{hold00}
\\
I_{0} \, \Exp_{0}[N] &= \log A  + \kappa_{0}^{1} + \log q_{0}^{1} +o(1); \label{hold0}
 \end{align}

\noindent {\rm (c)} for $r>1$, 
\begin{align}
I_{0} \, \Exp_{0}[M] &= \log A+   2 \, d_{r} \, \sqrt{\log A }  + O(1), \label{zhang00} 
\\
I_{0} \, \Exp_{0}[N] &= \log A+   2 \, d_{r} \, \sqrt{\log A }   + O(1) , \label{zhang0}
\end{align}
where $d_{r}$ is defined in \eqref{drhr}.
\end{theorem}
%%%%%%%%%%%%%%

%%%Proof
\begin{proof}
(a) Asymptotic approximations (\ref{highM}) and (\ref{highN}) can be relatively easily established using nonlinear renewal theory. Specifically, starting from representation (\ref{Z}) and applying the Nonlinear Renewal Theorem (see Theorem~9.28 in \cite{sieg-book}), it can be shown (as in Theorem~2.1 of \cite{feltar12}) that $I_{i} \, \Exp_{i}[M_{B}^{1}]$ is equal to the right-hand side of (\ref{highM}) as $B \rightarrow \infty$. Therefore, to prove  (\ref{highM}) it suffices to show that $\Exp_{i}[M_{B}^{1}-M]=o(1)$  as $A, B \rightarrow \infty$, or equivalently as $\alpha, \beta \rightarrow 0$. To this end, note that
$$
0 \leq M_{B}^{1}-M = [M_{B}^{1}-M_{A}^{0}] \,\Ind{d_{M}=0} \leq M_{B}^{1} \; \Ind{d_{M}=0} .
$$
Applying the  Cauchy--Schwartz inequality,  we obtain
\begin{equation}  \label{qazx}
\Exp_{i}[M_{B}^{1} \; \Ind{d_{M}=0}] \leq  \sqrt{\Exp_{i}[(M_{B}^{1})^{2}] \; \Pro_{i}(d_{M}=0)}.
\end{equation}
From (\ref{Z}) and (\ref{Y}) it is clear that $Z_{t}(\qb_{1}) \geq Z_{t}^{i}+\log q_{1}^{i}$, $ t \in \mathbb{N}$, thus,
$$
M_{B}^{1} \leq  \inf\{t: Z_{t}^{i} \geq \log (B/q_{1}^{i}) \}.
$$
Consequently, from Theorem 8.1 in \cite{gut} it follows that, since (A1) holds, 
$$
(I_{i})^{2} \,  \Exp_{i}[(M^{1}_{B})^{2}] \leq (\log (B/q_{1}^{i}))^{2} (1+o(1)).
$$
From the latter inequality and Lemma~\ref{newlemma} we conclude that 
$$
\Exp_{i}[(M^{1}_{B})^{2}]= O((\log B)^{2})=O(|\log \alpha|^{2}).
$$
Moreover, since (A3)  implies $\Pro_{i}(d_{M}=0)\leq \beta (1+o(1))$,  (\ref{qazx}) becomes 
$$
\Exp_{i}[M_{B}^{1} \; \Ind{d_{M}=0}] =  O(|\log \alpha|^{2} \beta)
$$ and from (A2) we conclude that the upper bound goes to 0. This completes the proof of (\ref{highM}), whereas the proof of (\ref{highN}) is analogous.

(b) From representation (\ref{Z}) and the Nonlinear Renewal Theorem it follows that $I_{0} \, \Exp_{0}[M_{A}^{0}]$ is equal to the right-hand side of (\ref{hold00}) as $A \rightarrow \infty$. Then, similarly to (a), we can show that  $\Exp_{0}[M_{A}^{0}-M]=o(1)$.
The proof of (\ref{hold0}) follows similar steps.  

(c) In order to prove (\ref{zhang0}), we start from representation (\ref{rip}) and apply nonlinear renewal theory of \citet{Zhang-AP88}. As a result,
it can be shown (analogously to Lemma~2.1 of  \citet{Drag}) 
that $I_{0} \, \Exp_{0}[N_{A}^{0}]$  is equal to the right-hand side of (\ref{zhang0}). Thus, it suffices to show that $\Exp_{0}[N_{A}^{0}] =\Exp_0 [N] +o(1)$, which can be done in just the same way as in (a) and (b).
\end{proof}
%%%%

%%Remark
\begin{remark}
Asymptotic approximation \eqref{zhang0} can be further improved (up to the negligible term $o(1)$), if stronger integrability conditions are postulated on the vector $W$ defined in (\ref{W}). Specifically, if in addition we assume the third moment condition $\Exp_{0}[||W||^{3}]<\infty$ as well as the Cramer-type condition $\lim \sup_{||t|| \rightarrow \infty} \Exp_{0}[e^{j <t,W>}] <1$, where $j$ is the imaginary unit,  $t=(t_{1},  \ldots, t_{r})$ and $<t,W>=\sum_{l=1}^{r} t_{l} W_{l}$, then the following expansion holds
\[
\begin{aligned}
I_{0} \, \Exp_{0}[N] & =  \log A+   2 \, d_{r} \, \sqrt{\log A + d_{r}^{2}}  +  \frac{h_{r}^{2}}{2 I_{0}}  + \kappa_{0}^{1} 
\\
& \quad +  \int_{\mathbb{R}^{r}} \left\{\max_{1 \leq i \leq r} (x_{i}) \; \left[\cP(x)+ \lambda(\qb_{0}) \, \Sigma^{-1}  x'\right]\right\} \;  \phi_{\Sigma}(x) \, dx + o(1),
\end{aligned}
\]
where  $\lambda(\qb_{0})= (\log q_{0}^{1}, \ldots , \log q_{0}^{r})$ and  $\cP$ is a third-degree polynomial whose coefficients depend on the $\Pro_{0}$-cumulants of $W$ (see \citet{bhatrao86}).  This approximation can be derived similarly to  Theorem 3.3 of \citet{dragalin-it00} based on nonlinear renewal theory of \citet{Zhang-AP88}.   
\end{remark}
%%%%%%%%%%%

%%%Cor
\begin{corollary}  \label{C3}
Suppose that {\rm (A1)--(A3)} hold with $k_{0}=1$, i.e., $A$ and $B$ are selected so that  
$ \Pro_0(d_M=1) \sim \alpha$  and $\Pro_0(d_N=1) \sim \alpha$. Then, 
\begin{align}
I_{i} \, \Exp_{i}[M] &= |\log \alpha|+ \log \Bigl( \sum_{j=1}^{K} q_{1}^{j} \gamma_{j}  \Bigr)+ \kappa_{i} - \log q^{i}_{1} +o(1), \label{perM}\\
I_{i} \, \Exp_{i}[N] &\leq |\log \alpha|+ \log \Bigl( \sum_{j=1}^{K} q_{1}^{j} \gamma_{j}  \Bigr)+ \kappa_{i} - \log q^{i}_{1} +o(1).  \label{perN}
\end{align}
\end{corollary}

%%Proof
\begin{proof}
From (\ref{pr1})--(\ref{pr3}) it follows that 
\begin{align*}
\log B &= |\log \Pro_0(d_M=1)|+  \log \Bigl( \sum_{j=1}^{K} q_{1}^{j} \gamma_{j}  \Bigr)+o(1), \\
\log B &\leq  |\log \Pro_0(d_N=1)|+  \log \Bigl( \sum_{j=1}^{K} q_{1}^{j} \gamma_{j}  \Bigr) +o(1).
\end{align*}
Moreover, from (\ref{lk0}) and the assumption that $k_{0}=1$ we have  
$$ |\log \Pro_0(d_M=1)|= |\log \alpha| +o(1) \quad \text{and} \quad |\log \Pro_0(d_N=1)|= |\log \alpha| +o(1).$$
From these two relationships and Theorem \ref{T1}(a) we obtain the desired result. 
\end{proof}

%%%%%%%%%%%%%%%%%%%%%%%%%%%%%%%%%%%%%%%%%%%%%%%%%%%%%%%%%%%%%%%%%%%

%%________________________________________________________
\section{Asymptotic Optimality Properties} \label{sec:AM} 

In this section, we establish the asymptotic optimality properties of the MiLRT and the WGLRT.

%%___________________________________
\subsection{Uniform asymptotic optimality}

First, we show that both tests  minimize the expected sample size  within an $O(1)$ term  (i.e., to second order) under every $\Pro_{i}$, $1 \leq i \leq K$ and
at least to first order under $\Pro_{0}$. 

%%%Theorem
\begin{theorem}  \label{C2}
Suppose that conditions {\rm (A1)--(A3)} hold and that $A,B$ are selected so that $\delta_{\mathrm{mi}}, \delta_{\mathrm{gl}} \in \ccab$. 

\begin{enumerate}
\item [(a)] For every $1 \leq i \leq K$,
\begin{align} 
\Exp_{i}[M] &= \inf_{\delta \in \ccab} \Exp_{i}[T] + O(1) , \label{sec1}\\
\Exp_{i}[N] &= \inf_{\delta \in \ccab} \Exp_{i}[T] + O(1).  \label{sec2}
\end{align}

\item [(b)] If $r=1$, then  
\begin{align}
\Exp_{0}[M] &= \inf_{\delta \in \ccab} \Exp_{0}[T] + O(1), \label{oM} \\
\Exp_{0}[N] &= \inf_{\delta \in \ccab} \Exp_{0}[T]  + O(1), \label{oN}
\end{align}
whereas if $r>1$,  
\begin{align}
 \Exp_{0}[M] &= \inf_{\delta \in \ccab} \Exp_{0}[T] \; (1+o(1)), \label{Exp0M}\\
 \Exp_{0}[N] &= \inf_{\delta \in \ccab} \Exp_{0}[T] \; (1+o(1)). \label{Exp0N}
 \end{align}
\end{enumerate}
\end{theorem}
%%%%%%%%%%%

%%Proof
\begin{proof}
(a)  From (\ref{ESSi_SPRT}) it  is clear that 
\begin{align} 
 I_{i} \, \inf_{\delta \in \ccab} \Exp_{i}[T] &\geq |\log \alpha| + O(1) , \label{LB1}
\end{align}
whereas from Theorem~\ref{T1}(a) and Lemma \ref{newlemma} it follows that 
$$I_{i} \, \Exp_{i}[M] =\log B+O(1)= |\log \alpha |+ O(1),$$
which proves (\ref{sec1}). The proof of (\ref{sec2}) is similar.

(b) From (\ref{ESSi_SPRT0}) it is clear that for every $ 1\leq i \leq K$  we have 
\begin{align} 
\inf_{\delta  \in \ccab} \Exp_{0}[T] & \geq \frac{|\log \beta|}{I_{0}^{i}} + O(1), \label{LB2}
\end{align}
thus, recalling from (\ref{order}) that $I_{0}=\min_{1 \leq i \leq K} I_{0}^{i}$, we obtain
\begin{align} 
\inf_{\delta  \in \ccab} \Exp_{0}[T] & \geq \frac{|\log \beta|}{I_{0}} + O(1). \label{LB2_1}
\end{align}
But from  Theorem~\ref{T1}(b) and Lemma~\ref{newlemma} it follows that 
\begin{equation} \label{messages}
I_{0} \, \Exp_{0}[M] =\left\{
\begin{array}{cl}
\log A+O(1)= |\log \beta |+ O(1), \quad &\mbox{if}~ r=1,  \\
\log A \,  (1+o(1))= |\log \beta | \, (1+o(1)), \quad &\mbox{if}~ r>1,
\end{array}
\right.
\end{equation}
which implies (\ref{oM}) and (\ref{Exp0M}). The proofs of (\ref{oN}) and (\ref{Exp0N}) are similar.
\end{proof}
%%%%%%%%%%%%

%%______________________________
\subsection{Almost optimality}

In what follows, we denote by  $\delta^{*}_{\rm mi}(\pb)=(M^*(\pb),d_{M^*(\pb)})$ and $\delta^{*}_{\rm gl}(\pb)=(N^*(\pb),d_{N*(\pb)})$ the MiLRT and the WGLRT 
with weights  given by (\ref{choose}), i.e.
\begin{equation} \label{choose2}
q_{1}^{i}= \frac{p_{i}}{\cL_{i}} \quad \text{and} \quad  q_{0}^{i}=p_{i} \, \cL_{i} , \quad 1  \leq i \leq K,
\end{equation}
where $\pb=(p_{1},  \ldots, p_{K})$, $p_{i}>0$ for every $1 \leq i \leq K$ and $\sum_{i=1}^{K} p_{i}=1$.  
Our goal is to show that $\delta^{*}_{\rm mi}(\pb)$ and $\delta^{*}_{\rm gl}(\pb)$ attain $\inf_{\delta \in \ccab} \Exp^{\pb}[T]$
asymptotically within an $o(1)$ term, where $\Exp^{\pb}$ is expectation with respect to the weighted probability measure $\Pro^{\pb}=\sum_{i=1}^{K} p_{i}  \, \Pro_{i}$. Before doing so, note that Corollary \ref{C3} implies that if $B$ is selected so that
$\Pro_0(d_{M^{*}(\pb)}=1) \sim \alpha$ and $\Pro_0(d_{N^{*}(\pb)}=1) \sim \alpha$,  then
\begin{align}  
 \Exp_{i}[M^{*}(\mathbf{p})] &= \frac{1}{I_{i}} \Bigl[|\log \alpha|+  \kappa_{i} + \log \gamma_{i} +  C_{i}(\mathbf{p}) \Bigr] +o(1),  \label{perM222} \\
 \Exp_{i}[N^{*}(\mathbf{p})] &\leq \frac{1}{I_{i}} \Bigl[|\log \alpha|+  \kappa_{i} + \log \gamma_{i} +  C_{i}(\mathbf{p}) \Bigr] +o(1), \label{perN222}
\end{align} 
where we have used the fact that $\cL_{i}= \gamma_{i} I_{i}$, $1\leq i\leq K$ and we have introduced the following notation
\begin{equation} \label{C}
C_{i}(\mathbf{p})=\log \Bigl( \sum_{j=1}^{K} \frac{p_{j}}{I_{j}} \Bigr) - \log \frac{p_{i}}{I_{i}}, \quad 1 \leq i \leq K.
\end{equation}

%%Theorem
\begin{theorem} \label{T2} 
Suppose that  conditions {\rm (A1)--(A3)} hold with $k=1$, i.e., $\alpha,\beta \rightarrow 0$ so that $|\log \alpha| \sim |\log \beta|$. Then
\begin{align} \label{t}
\inf_{\delta \in \ccab} \Exp^{\pb}[T] &=  \sum_{i=1}^{K} \frac{p_{i}}{I_{i}} \Bigl[|\log \alpha|+  \kappa_{i} + \log \gamma_{i} +  C_{i}(\mathbf{p}) \Bigr] +o(1). 
\end{align} 
Moreover, if $A,B$ are selected so that  $\delta^{*}_{\rm mi}(\pb)$  and $\delta^{*}_{\rm gl}(\pb)$ belong to $\ccab$ and $k_{0}=1$, i.e.,
$\Pro_0(d_{M^{*}(\pb)}=1) \sim \alpha$ and $\Pro_0(d_{N^{*}(\pb)}=1) \sim \alpha$, then 
 \begin{align*}
\inf_{\delta \in \ccab} \Exp^{\pb}[T] &=  \Exp^{\pb}[M^{*}(\pb)] +o(1), \\ 
\inf_{\delta \in \ccab} \Exp^{\pb}[T] &=  \Exp^{\pb}[N^{*}(\pb)] +o(1).
\end{align*} 
\end{theorem}
 %%%%%%%%%%

In order to prove this theorem, we formulate our sequential testing problem as a Bayesian sequential decision problem with  $K+1$ states, $\Hyp_{0}: f=f_{0}$ and $\Hyp_{1}^{i}: f=f_{i}$, $1 \leq i \leq K$
and  two possible actions upon stopping, either accepting $\Hyp_{0}$ or $\Hyp_{1}=\cup_{i} H_{1}^{i}$. Moreover, we denote by $c$ the sampling cost per observation and  by $w_{1}$ (resp. $w_{0}$) the loss associated with accepting $\Hyp_{0}$ (resp. $\Hyp_{1}$) when the correct hypothesis is $\Hyp_{1}$ (resp. $\Hyp_{0}$).  We also define the probability measure $\Prop = \pi \, \Pro_{0} + (1-\pi) \, \Pro^{\pb}$, which means that $\pi=\Prop(\Hyp_{0})$ is the prior probability of $\Hyp_{0}$ and $p_{i}=\Prop(\Hyp_{1}^{i}|\Hyp_{1})$ is  the prior probability of $f=f_{i}$ given that $\Hyp_{1}$ is correct.

 The integrated risk of a sequential  test $\delta=(T,d_{T})$ is defined as the sum $\cR(\delta)=\cR_{c}(T)+ \cR_{s}(d_{T})$, where 
$\cR_{c}(T)$ is the integrated risk due to sampling and $\cR_{s}(d_{T})$  is the integrated risk due to a wrong decision upon stopping, i.e., 
 \begin{align*} 
 \cR_{c}(T) &= c  \, \Expop[T]=  c \Bigl[ \pi \, \Exp_{0}[T] + \, (1-\pi) \,  \Exp^{\pb}[T] \Bigr], \\
\cR_{s}(d_{T}) &=  \Expop[ w_{0} \,\Ind{d_{T}=1} | \Hyp_{0}] +  \Expop[w_{1} \, \Ind{d_{T}=0}  | \Hyp_{1}] \\
 &= \pi \, w_{0} \, \Pro_{0}(d_{T}=1) +  (1-\pi) \, w_{1} \,  \Pro^{\pb}(d_{T}=0) .
\end{align*} 
The Bayesian sequential decision problem  is to find an optimal (\textit{Bayes}) sequential test that attains the  \textit{Bayes risk}, $\cR^{*}=\inf_{\delta} \cR(\delta)$.  It is well known that the solution to this problem does not have a simple structure (see, e.g., \cite{ChowRobbinsSiegmund-book71}). However, from the seminal work of \citet{Lorden-AS77}  on  finite-state sequential decision making it follows that $\delta_{\rm mi}^{*}(\mathbf{p})$ and  $\delta^{*}_{\rm gl}(\mathbf{p})$ are \textit{almost Bayes} when the thresholds $A$ and $B$ are chosen as 
\begin{equation} \label{qq}
A_{c}= \frac{1-\pi}{\pi} \frac{w_{1}}{c} \quad \text{and} \quad B_{c}= \frac{\pi}{1-\pi} \frac{w_{0}}{c} .
\end{equation} 
More specifically, denote by  $\delta^{*}_{\rm mi,c}(\pb) =(M_{c}^{*}(\pb), d_{M_{c}^{*}(\pb)})$ and  
$\delta^{*}_{\rm gl,c}(\pb)=(N_{c}^{*}(\pb), d_{N_{c}^{*}(\pb)})$  the sequential tests $\delta^{*}_{\rm mi}(\pb)$ and $\delta^{*}_{\rm gl}(\pb)$ when the thresholds are given by  $A_{c}$ and $B_{c}$. Under the integrability condition (A1), it follows from \citet{Lorden-AS77} that 
\begin{equation} \label{L4}
\cR(\delta^{*}_{\rm mi,c}(\pb))-\cR^{*} =o(c) \quad \text{and} \quad
\cR(\delta^{*}_{\rm gl,c}(\pb)) -\cR^{*} =o(c).
\end{equation}
The proof of Theorem \ref{T2} relies on this third-order Bayesian asymptotic optimality property, which requires symmetric thresholds (\ref{qq}) and is the reason why
we assumed in Theorem \ref{T2} that error probabilities go to 0 with the same rate.

%%Proof
\begin{proof}
In order to lighten the notation, we omit the dependence on the prior distribution $\pb$ and write simply
$\delta^{*}_{\rm mi}=(M^{*},d_{M^{*}})$ and $\delta^{*}_{\rm mi,c}=(M^{*}_{c},d_{M^{*}_{c}})$  instead of $\delta^{*}_{\rm mi}(\pb)=(M^{*}(\pb),d_{M^{*}(\pb)})$
and $\delta^{*}_{\rm mi,c}(\pb)=(M_{c}^{*}(\pb),d_{M_{c}^{*}(\pb)})$ (and similarly for the WGLRT). 

From Corollary \ref{C3} it is clear that the right-hand side in (\ref{t}) is attained by $\delta^{*}_{\rm mi }$ and $\delta^{*}_{\rm gl}$ when their thresholds are selected so that $\Pro_0(d_{M^{*}}=1) \sim \alpha$ and $\Pro_0(d_{N^{*}}=1) \sim \alpha$. If additionally  $\delta^{*}_{\rm mi}, \delta^{*}_{\rm gl} \in \ccab$, then $\inf_{\delta \in \ccab} \Exp^{\pb}[T]$ is attained by these two tests to  within an $o(1)$ term. Thus, it suffices to establish  (\ref{t}). 

Consider the class of sequential tests
$$\ccab^{\pb}=\{\delta: \Pro_{0}(d_{T}=1) \leq \alpha \quad \text{and} \quad \Pro^{\pb}(d_{T}=0) \leq \beta\}.$$
Since $\ccab\subset \ccab^{\pb}$, we have $\inf_{\delta \in \ccab} \Exp^{\pb}[T] \geq \inf_{\delta \in \ccab^{\pb}} \Exp^{\pb}[T]$. Thus, it suffices to show that
\begin{equation} \label{ok}
\inf_{\delta \in \ccab^{\pb}} \Exp^{\pb}[T] = \sum_{i=1}^{K} \frac{p_{i}}{I_{i}} \Bigl[|\log \alpha|+  \kappa_{i} + \log \gamma_{i} +  C_{i}(\mathbf{p}) \Bigr] +o(1).
\end{equation}
Consider now the sequential test $\delta^{*}_{\rm mi,c}=(M_c^*, d_{M_c^*})$ with thresholds $A_{c}$ and $B_{c}$ selected so that $\Pro_{0}(d_{M_{c}^{*}}=1) =\alpha$ and $\Pro^{\pb}(d_{M_{c}^{*}}=0) = \beta$. From Corollary \ref{C3} it is clear that   $\Exp^{\pb}[M_{c}^{*}]$ is equal to the right-hand side in (\ref{ok}) as $c \rightarrow 0$, which means that it  suffices to show that
$$
\inf_{\delta \in \ccab^{\pb}} \Exp^{\pb}[T]= \Exp^{\pb}[M^{*}_{c}]+o(1),
$$
where $o(1)$ is an asymptotically negligible term as $c  \rightarrow 0$. 
More specifically, if  $\delta$ is an arbitrary sequential test in $\ccab^{\pb}$, we need to show that, for sufficiently small $c$,
$| \Exp^{\pb}[T]- \Exp^{\pb}[M^{*}_{c}]|$ is bounded above by an arbitrarily small, but fixed number. 

First of all, we observe that
\begin{align} \label{ers}
\cR_{s}(d_{T}) &=  \pi \, w_{0} \, \Pro_{0}(d_{T}=1) + (1-\pi) \, w_{1} \, \Pro^{\pb}(d_{T}=0) \nonumber\\
&\leq  \pi \, w_{0} \, \alpha + (1-\pi) \, w_{1} \, \beta =  \cR_{s}(d_{M_{c}^{*}}),
\end{align} 
where the inequality is due to $\delta \in \ccab^{\pb}$ and the second equality follows from the 
assumption that  $\Pro_{0}(d_{M_{c}^{*}}=1)= \alpha$ and $\Pro^{\pb}(M_{c}^{*}=0)=\beta$.

From  (\ref{pr00})--(\ref{pr000}) and the definition of $A_{c}$ and $B_{c}$ in (\ref{qq}) we have  
\begin{align} \label{beg}
\cR_{s}(d_{M_{c}^{*}}) &= \pi \, w_{0} \, \Pro_{0}(d_{M_{c}^{*}}=1) + (1-\pi) \, w_{1} \, \Pro^{\pb}(d_{M_{c}^{*}}=0) \nonumber\\
&\leq  \pi \, w_{0} \, \frac{|\qb_{1}|}{B_{c}} + (1-\pi) \, w_{1} \,\sum_{i=1}^{K} p_{i}  \frac{1}{A_{c} q_{0}^{i}} \nonumber\\
&\leq  |\qb_{1}| (1-\pi) c  + \sum_{i=1}^{K} p_{i} \frac{\pi c}{q_{0}^{i}} \leq (Q-1) c,
\end{align}
where $Q >1$ is some constant that does not depend on $c$ or $\pi$.

Fix $\epsilon>0$ and introduce the following sequential test
$$
T_{\epsilon c}= \min \{M^{*}_{\epsilon c} ,  T \} \;, \quad d_{T_{\epsilon c}}= d_{T} \Ind {T\leq M^{*}_{\epsilon c}}+ d_{M^{*}_{\epsilon c}} \Ind {T> M^{*}_{\epsilon c}}.
$$
Obviously, 
\begin{align} \label{ers2}
\cR_{s}(d_{T_{\epsilon c}}) \leq \cR_{s}(d_{T}) + \cR_{s}(d_{M^{*}_{\epsilon c}}) & \leq \cR_{s}(d_{M^{*}_{c}}) + \cR_{s}(d_{M^{*}_{\epsilon c}}) \nonumber\\
                                                                                 & \leq \cR_{s}(d_{M^{*}_{c}}) +(Q-1) c\, \epsilon ,
\end{align}
where the first inequality is due to (\ref{ers}) and the second one is due to   (\ref{beg}).

Since $M^{*}_{c}$ is almost Bayes (recall \eqref{L4}), for all sufficiently small $c$
\begin{equation} \label{opt}
\cR_{c}({M^{*}_{c}}) + \cR_{s}(d_{M^{*}_{c}}) \leq  \cR_{c}(T_{\epsilon c})+ \cR_{s}(d_{T_{\epsilon c}}) +  c \, \epsilon.
\end{equation}
Then, from (\ref{ers2})  we obtain $\cR_{c}(M^{*}_{c})  \leq  \cR_{c}(T_{\epsilon c})+ Q \, c \, \epsilon$, and consequently,
\begin{align} \label{set}
 \pi \,  \Exp_{0}[M^{*}_{c}] +  (1-\pi) \, \Exp^{\pb}[M^{*}_{c}] &\leq \pi \,  \Exp_{0}[T_{\epsilon c}] + (1-\pi) \;  \Exp^{\pb}[T_{\epsilon c}] + 
Q \,   \epsilon  \nonumber\\
 &\leq \pi \,  \Exp_{0}[M^{*}_{\epsilon c}] + (1-\pi) \;  \Exp^{\pb}[T]  +  Q \,   \epsilon,
\end{align}
where the second inequality follows from the definition of $T_{\epsilon c}$. Rearranging terms, we obtain from (\ref{set}) that
\begin{align} \label{set2}
 \Exp^{\pb}[M^{*}_{c}]  -  \Exp^{\pb}[T] \leq  \frac{\pi}{1-\pi} \, \Bigl(\Exp_{0}[M^{*}_{\epsilon c}]- \Exp_{0}[M^{*}_{c}] \Bigr) + \frac{Q \, \epsilon}{1-\pi}.
\end{align}
Since the last inequality holds for any $\pi \in (0,1)$, we can set $\pi=\epsilon /(1+\epsilon)$, which implies 
$B_{c}= \epsilon w_{0}/c$ and $A_{c}= w_{1} / (\epsilon \, c)$, whereas  (\ref{set2}) becomes
\begin{equation} \label{set3}
 \Exp^{\pb}[M^{*}_{c}]  -  \Exp^{\pb}[T] \leq  \epsilon \, (\Exp_{0}[M^{*}_{\epsilon c}]- \Exp_{0}[M^{*}_{c}]) + Q \, \epsilon (1+\epsilon).
 \end{equation}
But from (\ref{hold00}) and (\ref{zhang0}) it follows that as $c \rightarrow 0$
$$I_{0} \, (\Exp_{0}[M^{*}_{\epsilon c}]- \Exp_{0}[M^{*}_{c}]) = O(\log A_{\epsilon c}- \log A_{c})$$
and from (\ref{qq}) we have  $\log A_{\epsilon c}- \log A_{c}$=$|\log \epsilon|+O(1)$ as $c \rightarrow 0$, which completes the proof. 
 \end{proof}

%%Remark
\begin{remark}
With a similar argument as the one used in the proof of Theorem \ref{T2} it can be shown that if 
$\Pro_{0}(d_{M^{*}(\pb)}=1) =\alpha$ and $\Pro^{\pb}(d_{M^{*}(\pb)}=0) = \beta$, then
$$
\inf_{\delta \in \ccab} \Exp_{0}[T]  \geq \inf_{\delta \in \ccab^{\pb}} \Exp_{0}[T]= \Exp_{0}[M^{*}(\pb)]+o(1)
$$
and similarly for $\delta_{\rm gl}$. However, the right-hand side in this asymptotic lower bound is generally not attained by $\delta^{*}_{\rm mi}(\pb)$ or  $\delta^{*}_{\rm gl}(\pb)$  when their  thresholds are selected so that $\delta_{\rm mi}$, $\delta_{\rm gl}$ $\in \ccab$.
\end{remark}

%%Remark
\begin{remark}
While we have no rigorous proof, we strongly believe that the assertions of Theorem~\ref{T2} (as well as of Theorem~\ref{T3} below) hold true in the more general case where $\alpha$ and $\beta$ approach zero in such a way that the ratio $\log \alpha / \log \beta$ is bounded away from zero and infinity, which allows one to cover the asymptotically asymmetric case as well.
\end{remark}

%%_________________________________________
\subsection{Almost minimaxity} \label{minimax}

For any stopping time $T$ and $1 \leq i \leq K$, we set $\Ic_i[T] = I_i \Exp_i[T]$. Without loss of generality, we restrict ourselves to $\Pro_{i}$-integrable
stopping times, thus, from Wald's identity it follows that 
$$
\Ic_i[T] = \Exp_i [Z_T^i]= \Exp_{i} \Bigl[ \log \frac{d \Pro_{i}}{d \Pro_{0}}\Big|_{\cF_{T}}\Bigr].
$$ 
In other words, $\Ic_i[T]$ is the expected KL divergence between $\Pro_{i}$ and $\Pro_{0}$ that is accumulated up to time $T$. Let $\mathbf{\hat{p}}=(\hat{p}_{1}, \ldots, \hat{p}_{K})$ denote the prior distribution for which
\begin{equation} \label{hat}
\hat{p_{i}} = \frac{\cL_{i} e^{\kappa_{i}}}{\sum_{j=1}^{K} \cL_{j} \, e^{\kappa_{j}}}, \quad 1 \leq i \leq K.
\end{equation}
Then, from  (\ref{highM})--(\ref{highN}) it follows that  $\hat{\pb}$ (almost) equalizes the KL-divergence that is accumulated by both the MiLRT and the WGLRT until stopping, in the sense that $\cI_{i}[M^{*}(\hat{\pb})]$ and $\cI_{i}[N^{*}(\hat{\pb})]$ are independent of $i$ up to an $o(1)$ term. Indeed, 
\begin{align}
\cI_{i}[M^{*}(\hat{\pb})] &= \log B+ \log \Bigl( \sum_{j=1}^{K} e^{ \cL_{j} \, \kappa_{j}} \Bigr) +o(1),  \\
\cI_{i}[N^{*}(\hat{\pb})] &= \log B+ \log \Bigl( \sum_{j=1}^{K} e^{ \cL_{j} \, \kappa_{j}} \Bigr)+ o(1),
\end{align}
where only negligible terms $o(1)$ may depend on $i$. If additionally $B$ is selected so that  $\Pro(d_{M^{*}(\hat{\pb})}=1) \sim \alpha$ and $\Pro(d_{N^{*}(\hat{\pb})}=1) \sim \alpha$, then  (\ref{perM})--(\ref{perN}) imply that for every $1 \leq i \leq K$,
\begin{align}
\cI_{i}[M^{*}(\mathbf{\hat{p}})] &= |\log \alpha| + \log \Bigl( \sum_{j=1}^{K} \gamma_{j} e^{\kappa_{j}} \Bigr) +o(1),  \label{newM} \\ 
\cI_{i}[N^{*}(\mathbf{\hat{p}})] & \leq |\log \alpha| + \log \Bigl( \sum_{j=1}^{K} \gamma_{j} e^{\kappa_{j}} \Bigr) +o(1),  \label{newN}
\end{align}
and consequently,  if we denote by $\hat{\cI}[T]= \max_{1 \leq i \leq K} \Ic_{i}[{T}]$ the maximal expected KL-divergence until stopping, we have
\begin{align}
\hat{\cI}[M^{*}(\mathbf{\hat{p}})] &= |\log \alpha| + \log \Bigl( \sum_{j=1}^{K} \gamma_{j} e^{\kappa_{j}} \Bigr) +o(1),  \label{newM2} \\ 
\hat{\cI}[N^{*}(\mathbf{\hat{p}})] & \leq |\log \alpha| + \log \Bigl( \sum_{j=1}^{K} \gamma_{j} e^{\kappa_{j}} \Bigr) +o(1).  \label{newN2}
\end{align}

The following theorem states that $\delta_{\rm mi}(\hat{\pb})$ and $\delta_{\rm gl}(\hat{\pb})$ are almost minimax in this KL-sense.

%%Theorem
\begin{theorem} \label{T3}
Suppose that  conditions {\rm (A1)--(A3)} hold with $k=1$, i.e., $\alpha,\beta \rightarrow 0$ so that $|\log \alpha| \sim |\log \beta|$. Then,  
\begin{align} \label{MLB}
\inf_{\delta  \in \ccab}  \hat{\Ic}[{T}] = |\log \alpha| + \log \Bigl( \sum_{j=1}^{K} \gamma_{j} e^{\kappa_{j}} \Bigr) +o(1).
\end{align}
If additionally $A,B$ are selected so that  $\delta_{\rm mi}(\hat{\pb})$, $\delta_{\rm gl}(\hat{\pb})$ $\in \ccab$ and $k_{0}=1$, i.e.,
$\Pro(d_{M^{*}(\hat{\pb})}=1) \sim \alpha$ and $\Pro(d_{N^{*}(\hat{\pb})}=1) \sim \alpha$, then 
 \begin{align}
\inf_{\delta  \in \ccab}  \hat{\Ic}[{T}] &=   \hat{\Ic}[M^{*}(\hat{\pb})] +o(1), \label{equM} \\
\inf_{\delta  \in \ccab}  \hat{\Ic}[{T}] &=   \hat{\Ic}[N^{*}(\hat{\pb})] +o(1).  \label{equN}
\end{align} 
\end{theorem}

%%%Proof
\begin{proof}
Suppose that thresholds $A$ and $B$ are selected so that $\delta_{\rm mi}(\hat{\pb}) \in \ccab$ and $\Pro(d_{M^{*}(\hat{\pb})}=1) \sim \alpha$. From Theorem \ref{T2} it follows that 
\begin{align} \label{jet1}
\sum_{i=1}^{K} \hat{p}_{i} \, \Exp_{i}[M^{*}(\mathbf{\hat{p}})] +o(1) &\leq \inf_{\delta \in \ccab}  \, \sum_{i=1}^{K} \hat{p}_{i} \,\Exp_{i}[T] = \inf_{\delta \in \ccab}  \sum_{i=1}^{K} \frac{\hat{p}_{i}}{I_{i}} \,\Ic_{i}[T]  \nonumber \\
&\leq \Bigl(\sum_{i=1}^{K} \frac{\hat{p}_{i}}{I_{i}} \Bigr)  \,  \inf_{\delta \in \ccab} \, \hat{\Ic}[T],
\end{align}
whereas from (\ref{newM}) and (\ref{newM2}) we have 
\begin{align}
 \sum_{i=1}^{K} \hat{p}_{i} \Exp_{i}[M^{*}(\mathbf{\hat{p}})] &=  
 \Bigl(\sum_{i=1}^{K} \frac{\hat{p}_{i}}{I_{i}} \Bigr) \, \Bigl[ |\log \alpha| + \log \Bigl( \sum_{j=1}^{K} \gamma_{j} e^{\kappa_{j}} \Bigr) +o(1) \Bigr]  \label{jet3} \\
 &= \Bigl(\sum_{i=1}^{K} \frac{\hat{p}_{i}}{I_{i}} \Bigr) \, \hat{\Ic}[M^{*}(\mathbf{\hat{p}})]. \label{jet4}
\end{align}
From (\ref{jet1}) and (\ref{jet3}) we obtain (\ref{MLB}), whereas from (\ref{jet1}) and (\ref{jet4}) we obtain (\ref{equM}). Finally, from (\ref{newN2}) and (\ref{MLB}) we obtain (\ref{equN}). 
\end{proof}

%%______________________________________________
\section{How to Select $\pb$?} \label{sec:p}

In this section, we consider the specification of the prior distribution $\pb$, which determines the weights $\qb_0$ and $\qb_1$ of the MiLRT and the WGLRT when the weights are selected according to (\ref{choose}). Our goal is to select a \textit{robust} prior, which inflicts a small performance loss 
under every scenario. In other words, we want to avoid a prior distribution that leads to sequential tests with very good behavior for some  densities in $\Hyp_1$, but with poor behavior for others.

\subsection{Performance measures}
We will quantify the ``performance loss'' of the MiLRT (and similarly for the WGLRT) under $\Pro_{i}$ by the following measure, 
\begin{align*}
\cJ_{i}(\pb) &= \frac{\Exp_{i}[M^{*}(\pb)]- \Exp_{i}[S^{i}]}{\Exp_{i}[S^{i}]} , \quad 1 \leq i \leq K, 
\end{align*}
where we recall that $S^{i}$ is the SPRT  for testing $f_{0}$ against $f_{i}$. That is, $\cJ_{i}(\pb)$ represents the \textit{additional} expected sample size 
due to the uncertainty in the alternative hypothesis divided by the smallest possible expected sample size that is required for testing $f_{0}$ against $f_{i}$. Moreover, if 
$S^{i}$ has error probabilities $\alpha$ and $\beta$,  assumptions (A1)--(A3) hold and $k_{0}=1$, then 
from (\ref{ESSi_SPRT}) and (\ref{perM222}) it follows that  
\begin{align} \label{app}
\cJ_{i}(\mathbf{p}) &\approx \frac{C_{i}(\pb)}{|\log \alpha|+  \kappa_{i} + \log \gamma_{i}} = \frac{\log \Bigl[ \sum_{j=1}^{K} (p_{j}/I_{j}) \Bigr] + \log I_{i} -\log p_{i}}{|\log \alpha|+  \kappa_{i} + \log \gamma_{i}}  ,
\end{align}
where by $\approx$ we mean that the two sides differ by an $o(1)$ term.  From this expression we can see that the magnitude of $\cJ_{i}(\pb)$ is mainly determined by $K$, the cardinality of $\cA_{1}$, and the probability of type-I error $\alpha$. In particular, 
for every $1 \leq i \leq K$ and  $\mathbf{p}$, $\cJ_{i}(\mathbf{p})$ will be ``small'' when $|\log \alpha|$ is much larger than $\log K$, which implies that 
the choice of $\pb$ may make a difference only when $|\log \alpha|$ is not much larger than  $\log K$. 

%%Table
\begin{table}[!h]
	\centering
		\caption{Asymptotic performance loss for different prior distributions}
		\begin{tabular}{|c|c|c|} \hline
	  $p_{i}$  & $q_{1}^{i}$ & $C_{i}(\pb)$      \\ \hline  \hline
	 $\cL_{i}$ &    1 &  $-\log(\gamma_{i}) + \log \Bigl( \sum_{j=1}^{K} \gamma_{j} \Bigr)$  \\ [2ex]  \hline 
	 $I_{i}$   &  $1/ \gamma_{i}$ &  $\log K$      \\ [2ex] \hline  
   $e^{\kappa_{i}} \cL_{i}$ &  $e^{\kappa_{i}}$ &   $-\log(\gamma_{i} \, e^{\kappa_{i}}) + \log \Bigl( \sum_{j=1}^{K} \gamma_{j} \, e^{\kappa_{j}} \Bigr)$    \\ [2ex] \hline
	   $1$   &  $1/\cL_{i}$ &  $\log(I_{i}) + \log \Bigl( \sum_{j=1}^{K} (1/I_{j}) \Bigr)$          \\ [2ex] \hline   
	 \end{tabular}
	\label{tab:1}
	\end{table}

Moreover, from (\ref{app}) it is clear that a good choice for $\pb$ would guarantee that $C_{i}(\pb)$ is ``small'' for every $1 \leq i \leq K$. In Table \ref{tab:1}, we present $C_{i}(\pb)$ for the almost least favorable distribution $\hat{\pb}$, defined in (\ref{hat}), as well as for some other intuitively appealing choices of $\pb$. In particular, we consider the priors $\pb^{I}$, $\mathbf{p^{\cL}}$, $\mathbf{p^{u}}$  which are defined so that 
$$p_{i}^{I} \propto I_{i}, \quad p_{i}^{\cL} \propto \cL_{i}, \quad p_{i}^{u} \propto 1,  \quad   1 \leq i \leq K.$$

Note that $\mathbf{p^{\cL}}$, $\mathbf{p^{\cI}}$, $\hat{\mathbf{p}}$ are ranked, in the sense that  
$\cL_{i} \leq I_{i} \leq e^{\kappa_{i}} \cL_{i}$, since $\cL_{i}= \gamma_{i} I_{i}$ and $\gamma_{i} \leq 1 \leq e^{\kappa_{i}} \gamma_{i}$.
Thus,  $\pb^{\cL}$ (resp. $\hat{\pb}$)  assigns relatively less (resp. more) weight than $\pb^{\cI}$ 
to a hypothesis as its ``signal-to-noise ratio'' increases. Note also that $\mathbf{p^{\cL}}$ and $\hat{\mathbf{p}}$ reduce to $\mathbf{p^{\cI}}$ 
when there is no overshoot effect, in which case $\kappa_{i}=0$ and $\gamma_{i}=1$, whereas all these three priors reduce to $\mathbf{p^{u}}$ in the symmetric case where $I_{i}$ and $\cH_{i}$ do not depend on $i$.

\subsection{Numerical comparisons}
	
In order to make some concrete comparisons, we focus on the multichannel setup (\ref{multisensor}), assuming that 
$\{g_{0}^{i},g_{1}^{i}\}$ can be embedded in a parametric family $g(x;\theta)$, so that 
\begin{equation} \label{parametric}
g_{0}^{i}(x)=g(x;\theta=0) \quad \text{and} \quad g_{1}^{i}(x)=g(x;\theta_{i}), \quad 1 \leq i \leq K,
\end{equation} 
where $\theta_{i}>0$ expresses the ``signal-to-noise ratio'' in channel $i$, $1 \leq i \leq K$.  

Consider the exponential model assuming that  
\begin{equation} \label{exponential}
g(x;\theta)= \frac{1}{1+\theta} e^{-x/(1+\theta)}, \; x>0.
\end{equation}
Then $I_{i}$, $\kappa_{i}$ and $\gamma_{i}$ take the following form
$$
I_{i}= \theta_{i}- \log(1+\theta_{i}), \quad \kappa_{i}= \theta_{i}, \quad \gamma_{i}=(1+\theta_{i})^{-1}.
$$
For the Gaussian model  $g(x)= \cN(x; \theta,1)$, where $\cN(x;\mu, \sigma)$ is  density of the normal distribution with mean $\mu$ and standard deviation $\sigma$, the above quantities become
\begin{align*}
 I_{i} &= \frac{\theta_{i}^{2}}{2}, \quad \gamma_{i}= \frac{1}{I_{i}} \exp\Bigl\{ -2 \, \sum_{n=1}^{\infty} \frac{1}{n} \, \Phi \Bigl( -\frac{\theta_{i}}{2} \, \sqrt{n} \Bigr) \Bigr\} , \\
\kappa_{i} &= 1 + \frac{\theta_{i}^{2}}{4}- i \, \sum_{n=1}^{\infty} \Bigl[ \frac{1}{\sqrt{n}} \, \phi \Bigl( \frac{\theta_{i}}{2} \, \sqrt{n} \Bigr) 
                  - \frac{\theta_{i}}{2} \, \Phi \Bigl( -\frac{\theta_{i}}{2} \, \sqrt{n} \Bigr)  \Bigr].
\end{align*}

Assume, for simplicity, that $\theta_{i}=4$ for  $1 \leq i \leq K/2$ and $\theta_{i}=\theta$ for $K/2 < i \leq K$. Thus, the ``expected'' signal  in the first (resp.\ last)  channel is stronger (resp.\ weaker) than the signal in the last (resp.\ first) channel  when $\theta<4$ (resp.\ $\theta>4$). 

Our goal is to evaluate $\cJ_{1}(\pb)$ and $\cJ_{K}(\pb)$, i.e., the inflicted  performance loss when signal is present in the first and last channel respectively, as a function of $\theta$, for different prior distributions. We do so using asymptotic approximation (\ref{app}), in which we have set $K=10$ and $\alpha=10^{-4}$, and we present the results for the exponential case in Figure~\ref{fig:1}  and for the Gaussian case  in Figure~\ref{fig:2}.

The plots in both figures show that setting $\pb=\mathbf{\hat{p}}$  (resp.\ $\pb=\mathbf{p^{u}}$) leads to a better performance when signal is present in the channel
with stronger (resp. weaker) signal-to-noise ratio. However, the inflicted performance loss when the signal is present in the other channel can be very high. On the other hand, setting $\pb=\mathbf{p}^{\mathcal I}$  or $\pb=\mathbf{p}^{\mathcal L}$ leads to a more robust performance, since the  performance loss is similar (and relatively small) irrespectively of the channel in which signal is present and of the relative signal strengths.

%%Figure
\begin{figure}[!h]
\centering
\includegraphics[width=0.4\linewidth, height=0.5\linewidth]{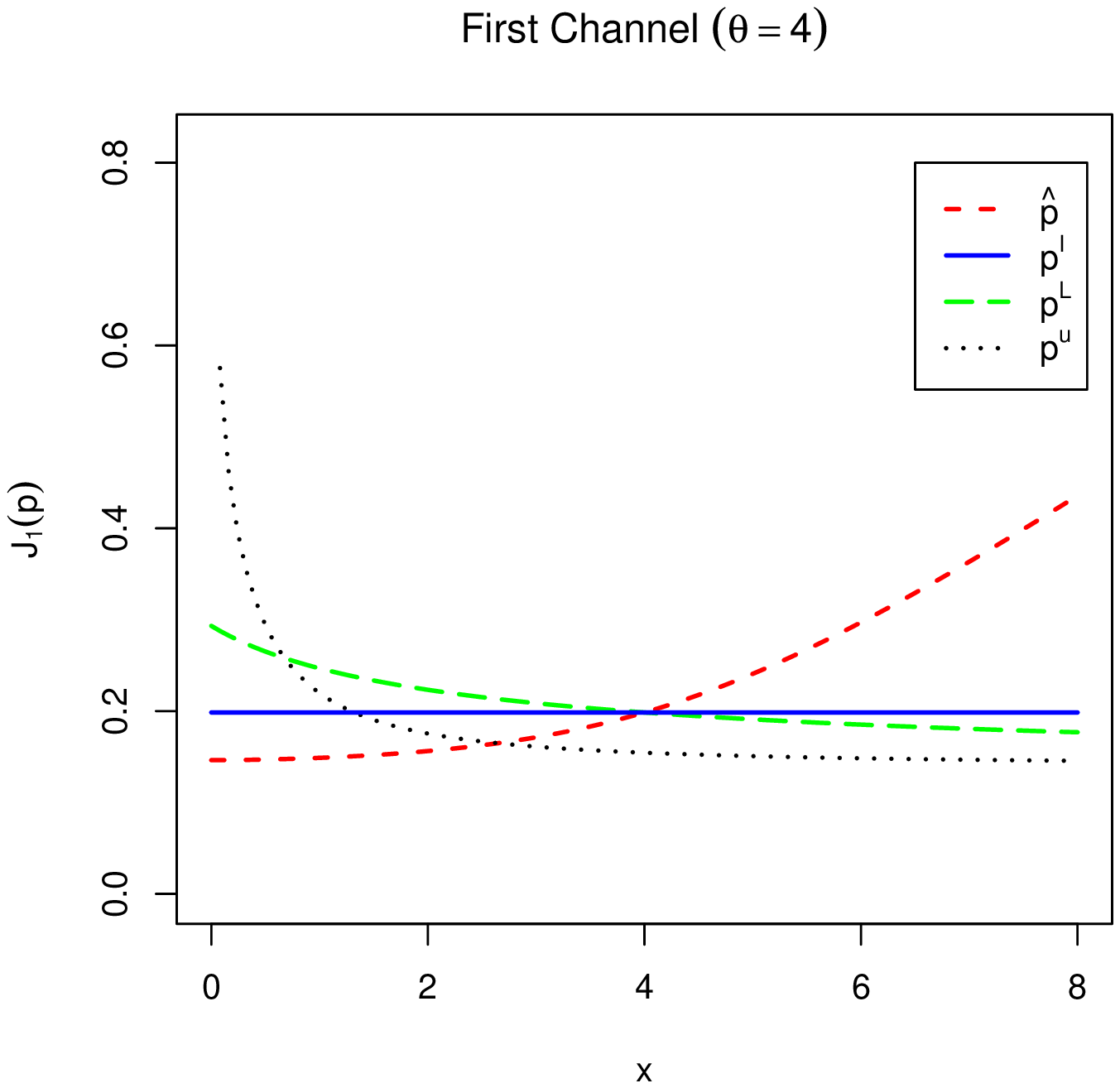}
\includegraphics[width=0.4\linewidth, height=0.5\linewidth]{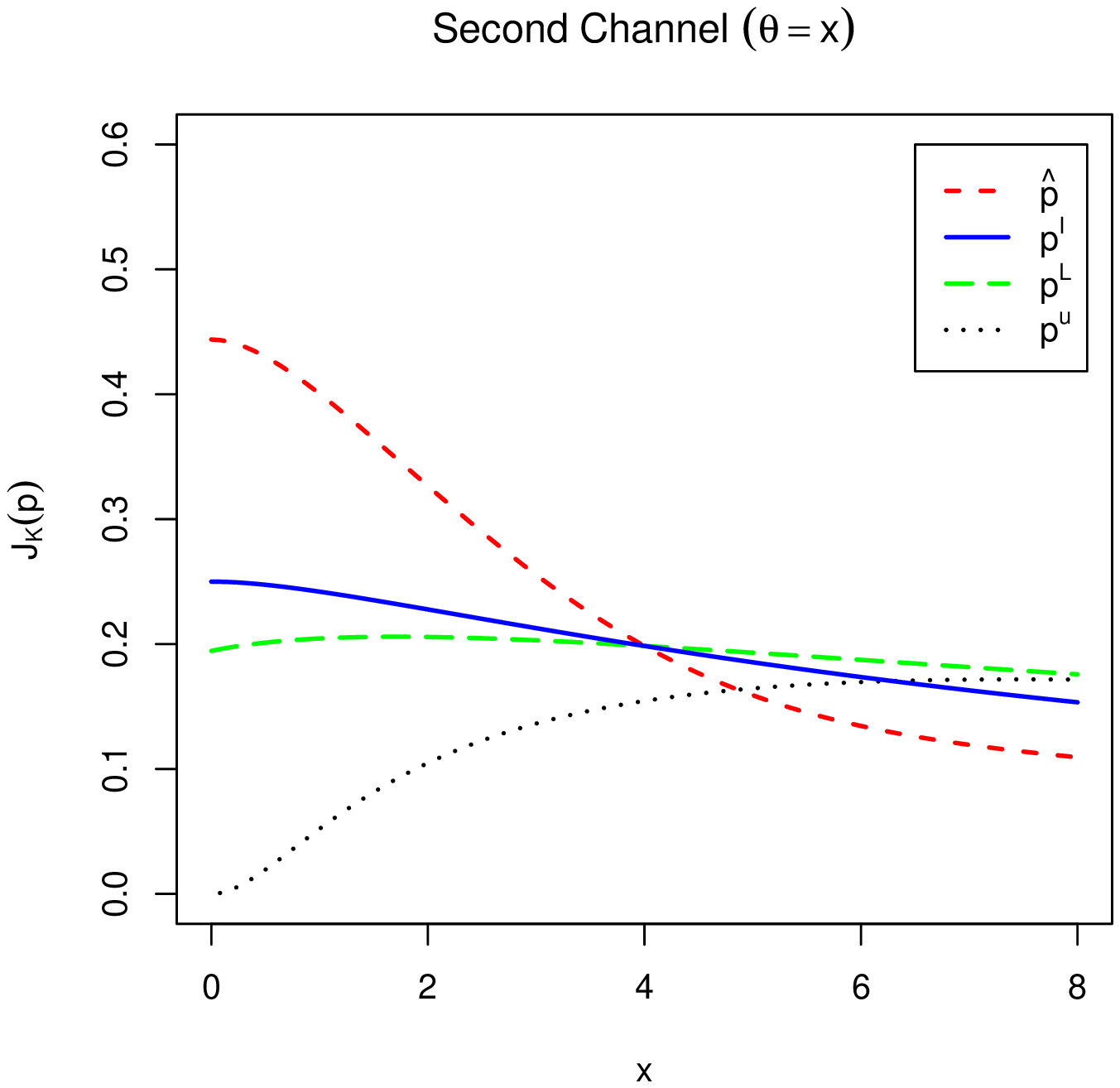}
\caption{Performance loss for different prior distributions in a multichannel problem with exponential data.}
\label{fig:1}
\end{figure}

%%Figure
\begin{figure}[!h]
\centering
\includegraphics[width=0.4\linewidth, height=0.5\linewidth]{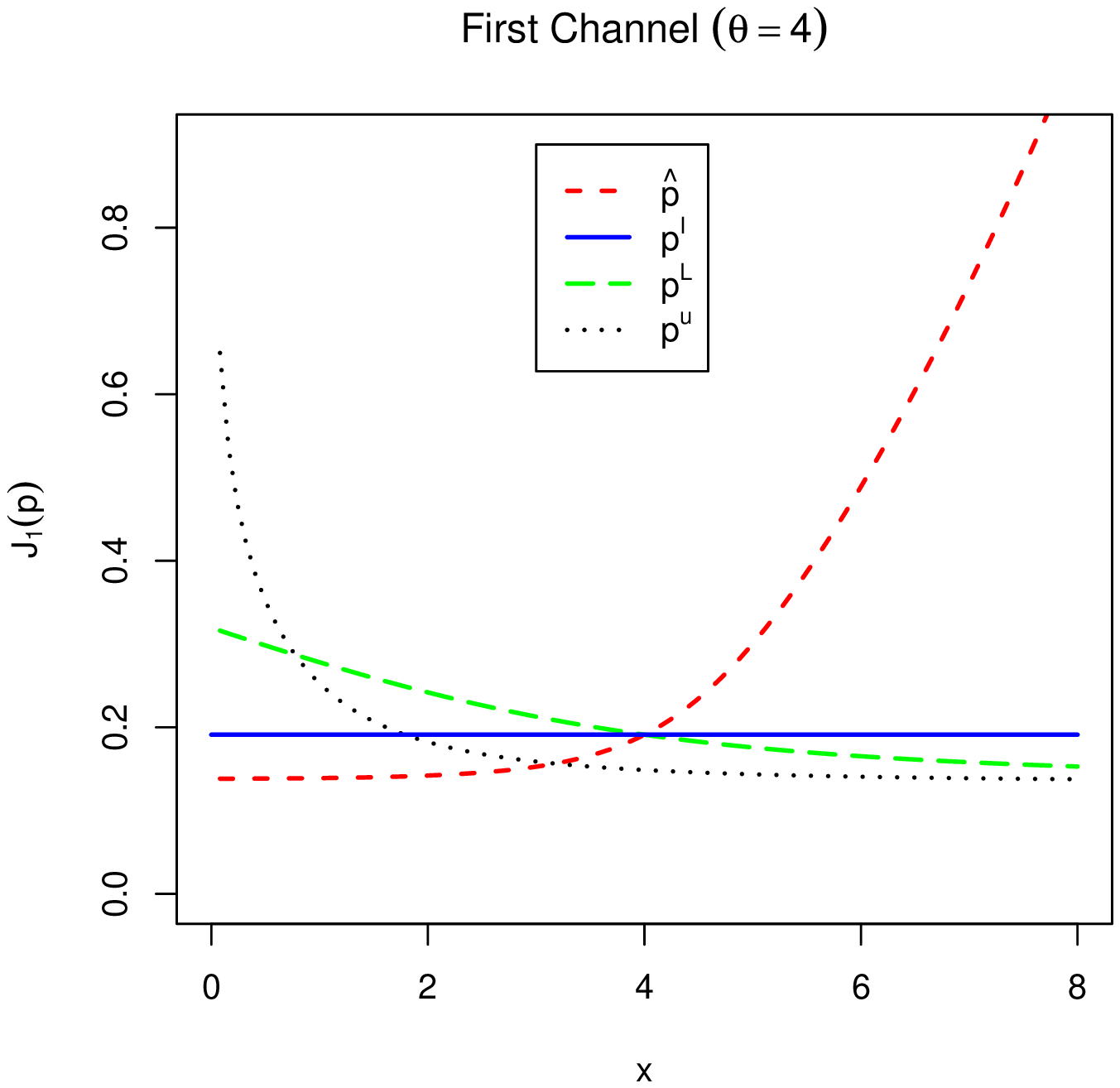} 
\includegraphics[width=0.4\linewidth, height=0.5\linewidth]{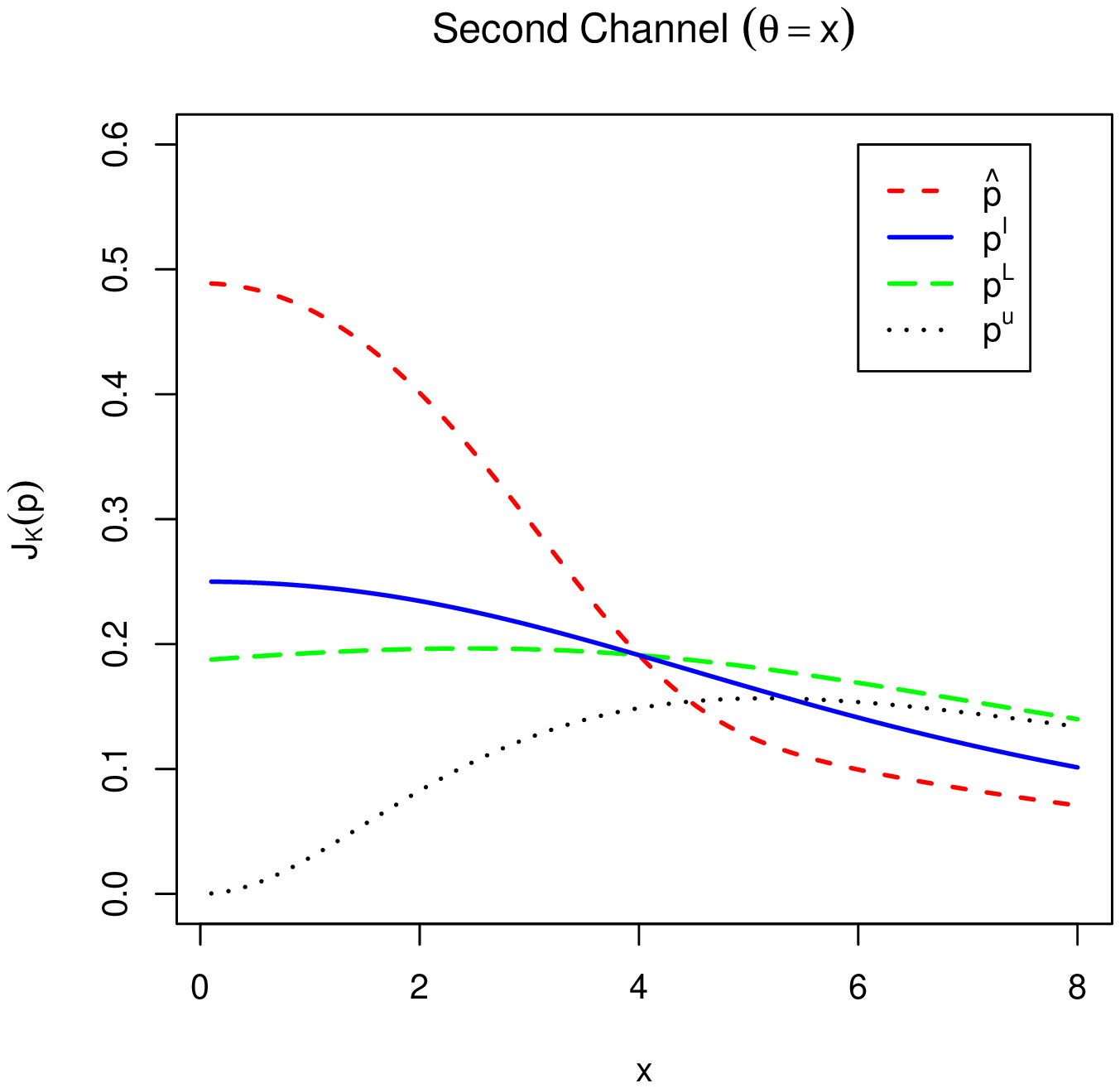}
\caption{Performance loss for different prior distributions in a multichannel problem with Gaussian data.} 
\label{fig:2}
\end{figure}

%%_________________________________________
\section{Monte Carlo Simulations}\label{sec:MC}

In this section, we present a simulation study whose goal is to check the accuracy of the asymptotic approximations established in Section \ref{sec:AM} and 
to compare the MiLRT with the WGLRT for realistic probabilities of errors. In particular, we consider the multichannel setup (\ref{multisensor}) with $K=3$ channels, 
exponential distributions given by (\ref{parametric})--(\ref{exponential}) and parameter values selected according to Table~\ref{tab:param}. Since our main emphasis is on the fast detection of signal, we set $\beta=10^{-2}$ and  consider different values of $\alpha$. Moreover, we choose the thresholds $A$ and $B$ according to (\ref{B}), whereas we select the weights according to (\ref{choose}) with $\pb=\mathbf{\pb}^{\mathcal I}$. 

%%%Table
\begin{table}[!h]
\centering
\caption{Parameter values in a multichannel problem with exponential data \label{tab:param}}
\begin{tabular}{c||c|c|c||c|c|c||c|c|}
$\theta_{i}$ & $I_{i}$ & $\kappa_{i}$ & $\gamma_{i}$  & $q_{1}^{i}$ & $q_{0}^{i}$  \\ \hline \hline
0.5 & 0.095   & 0.5 & 0.67 &   0.308  & 0.013   \\
1   & 0.584	  & 1   & 0.4  & 0.837  & 0.078    \\ 
2   & 0.901   & 2   & 0.33 & 1.380  & 0.138    \\
\end{tabular}
\end{table}
%--------------------------------------------------------------------------------

In the first three columns of Table~\ref{tab:2} we compare the type-I error probabilities for the two tests, which have been computed based on simulation experiments, against the target level $\alpha$. More specifically, these error probabilities are computed using representations (\ref{is1}) and (\ref{is2}) and \textit{importance sampling}, a simulation technique whose application in Sequential Analysis goes back to \citet{sieg-AS76}.  These results indicate that 
selecting $B$ according to (\ref{choice}) leads to type-I error probabilities very close to $\alpha$ for both tests,  even for relatively large $\alpha$. In particular, we see that $\Pro_{0}(d_{M^{*}} = 1)$ is slightly larger than $\alpha$,  which is expected, since (\ref{choice}) implies 
$\Pro_{0}(d_{M^{*}} = 1) \sim \alpha$, whereas we also observe that $\alpha$ is a sharp upper bound for  $\Pro_{0}(d_{N^{*}} = 1)$, 
the type-I error probability of the WGLRT.

%%%Table
\begin{table}[h]
\centering
\caption{Type-I error probabilities and the expected sample sizes under $\Pro_{i}$, $i=1,2,3$ for different values of the target probability
$\alpha$ when $\beta=10^{-2}$. \label{tab:2}}
\begin{tabular}{|c||c|c|c|c|c||c|c|c|} \hline
$\alpha$ & $\frac{\Pro_{0}(d_{M^{*}} = 1)}{\alpha}$ & $\frac{\Pro_{0}(d_{N^{*}} = 1)}{\alpha} $& $\Exp_{1}[M^{*}]$ & $\Exp_{1}[N^{*}]$ & $\Exp_{2}[M^{*}]$ & $\Exp_{2}[N^{*}]$ & $\Exp_{3}[M^{*}]$ & $\Exp_{3}[N^{*}]$ \\ \hline \hline
$10^{-2}$ & 1.051 &  0.994 & 59.9  & 59.4  &  17.8  & 19.4  & 6.2  &  7.3 \\
$10^{-3}$ & 1.033 &  0.995 & 84.1  & 84.1  &  25.7  & 27.1  &  9.0  &  9.9  \\
$10^{-4}$ & 1.025 &  0.996 & 108.5 & 108.3 &  33.7 & 34.6  &  11.7  &  12.4   \\ 
$10^{-5}$ & 1.017 &  0.996 & 132.5 & 132.3 &  41.4 & 42.0 & 14.3 &  15.0   \\ \hline
\end{tabular}
\end{table}

In the remaining columns of Table~\ref{tab:2}, we present for both tests the (simulated) expected sample size under $\Pro_{i}$, $i=1,2,3$ and in Figure  \ref{fig:3} we plot these values against the corresponding (simulated) type-I error probabilities. In these graphs, we also superimpose 
asymptotic approximation (\ref{perM}) (dashed line), as well as the asymptotic performance of the corresponding SPRT, (\ref{ESSi_SPRT}), which is given by the 
solid line. Triangles correspond to the WGLRT and circles to the MiLRT. From these results we can see, first of all, that asymptotic approximation  (\ref{perM}) is very accurate for both tests. Moreover, we can see that 
the two tests have similar performance. In particular, their performance is identical when signal is present in the channel with the smallest signal strength. In the other two cases, the MiLRT seems to perform slightly better, however the difference is small.

%%%Figures
\begin{figure}[!h]
\centering
\includegraphics[width=0.4\linewidth, height=0.5\linewidth]{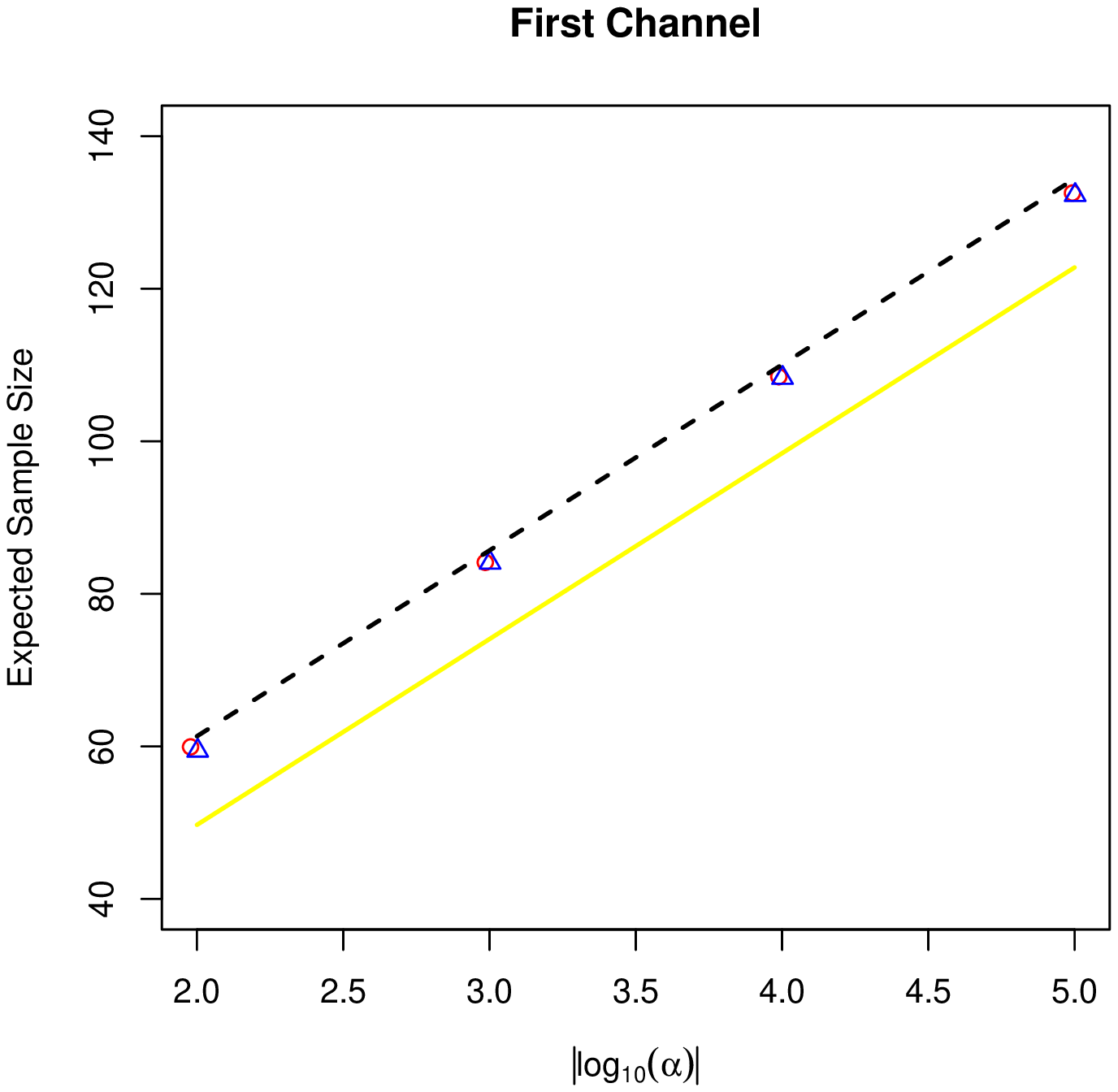}
\includegraphics[width=0.4\linewidth, height=0.5\linewidth]{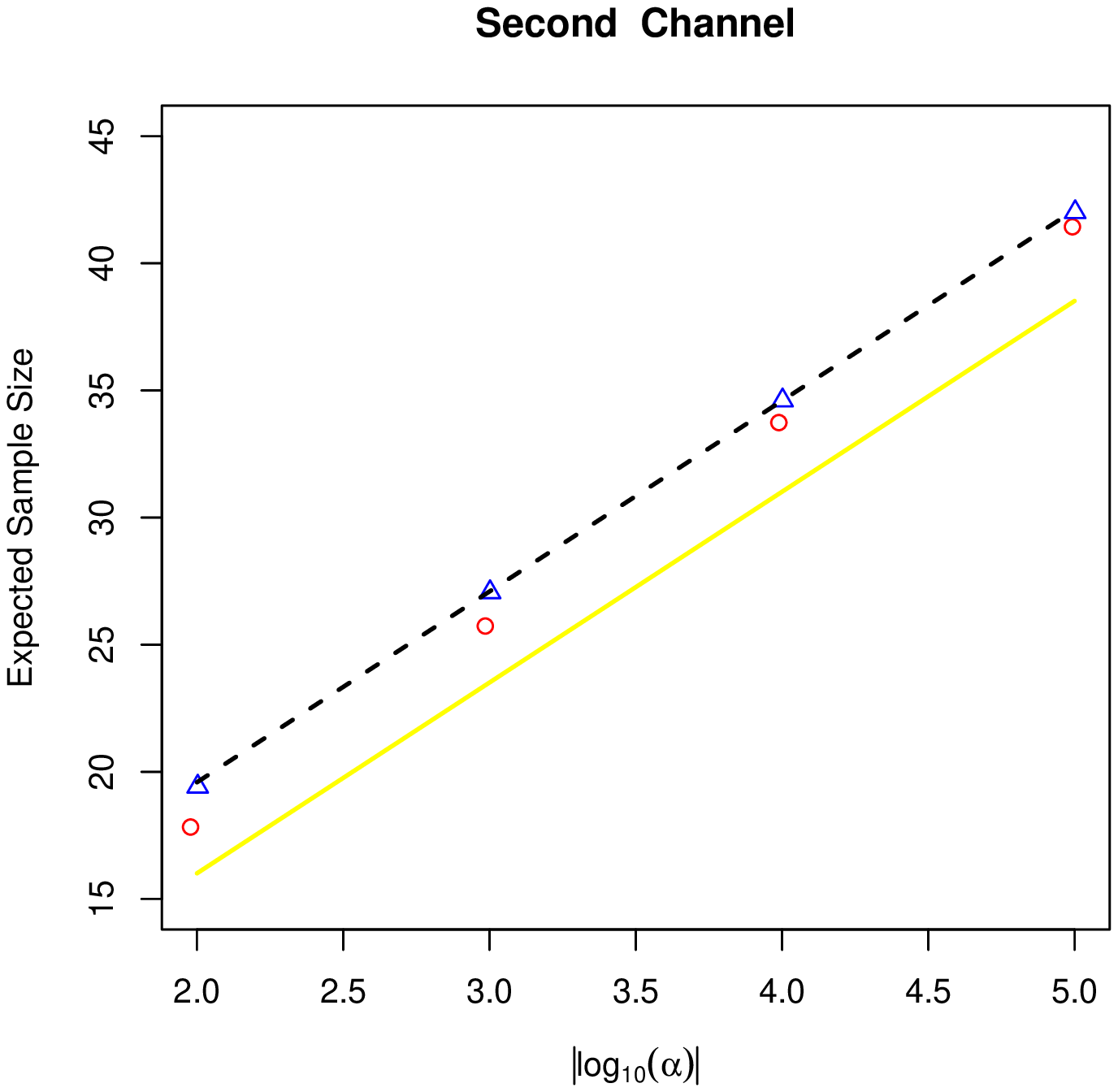}
\includegraphics[width=0.4\linewidth, height=0.5\linewidth]{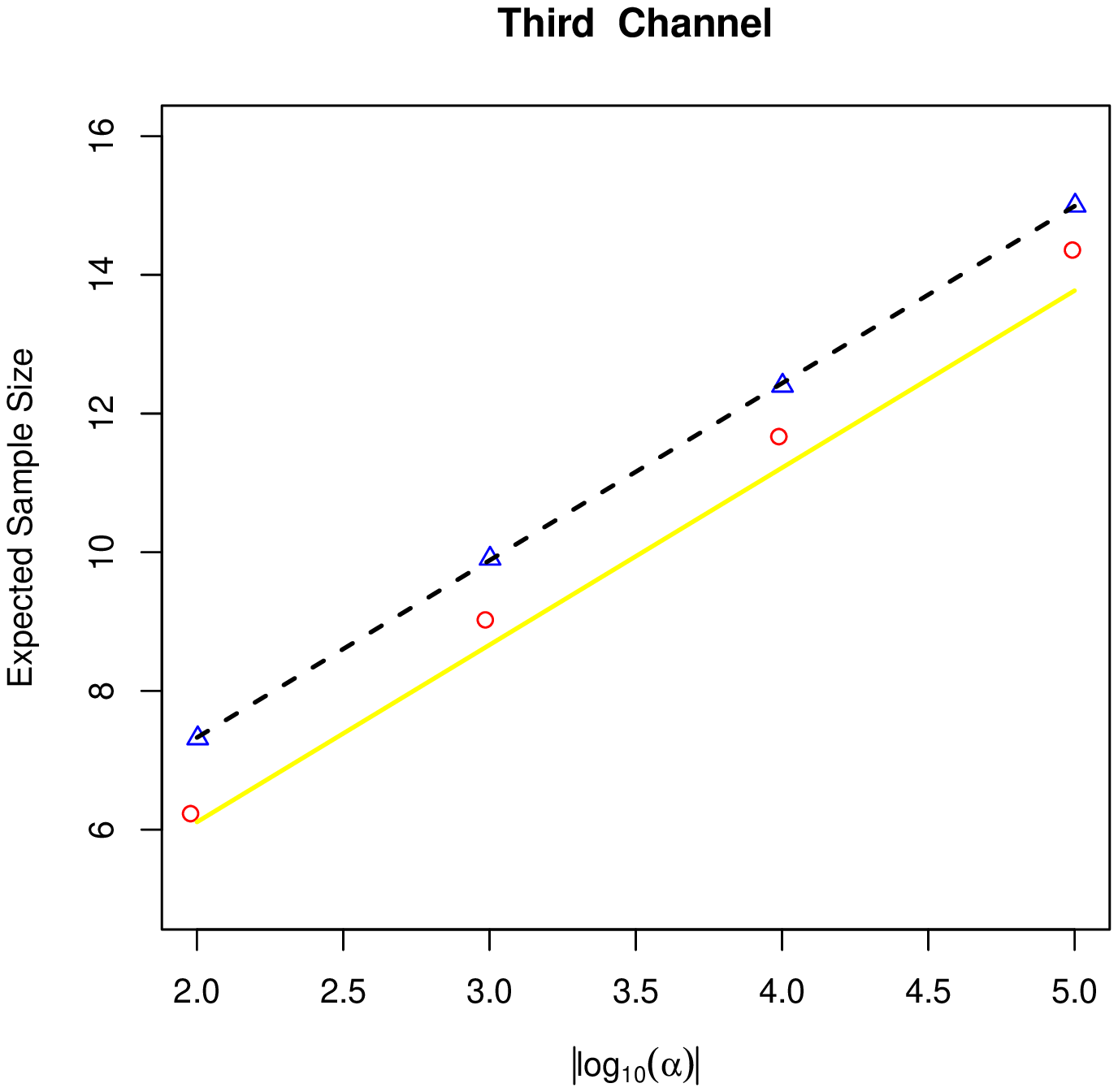} 
\label{fig:3}
\caption{Expected sample size of MiLRT and WGLRT under $\Pro_{i}$ against type-I error probability (in logarithmic scale), $i=1,2,3$. The dashed line represents asymptotic approximation (\ref{perM}), whereas the solid line refers to (\ref{ESSi_SPRT}), the asymptotic performance of the corresponding SPRT. The triangles (resp.\ circles) represent the simulated performance of the WGLRT (resp.\ MiLRT).}
\end{figure}

%--------------------------------------------------------------------------------
 \section{Conclusion}\label{sec:Concl}

In this work, we performed a detailed analysis and optimization of weighted GLR and mixture-based sequential  tests when the null hypothesis is simple and the alternative hypothesis is composite but discrete. Irrespectively of the choice of weights, both tests minimize asymptotically, at least to first order and often to second order, the expected sample size under each possible scenario as error probabilities go to 0.
However, with appropriate selection of weights, both test achieve higher-order asymptotic optimality properties. Specifically, they minimize a weighted expected sample size as well as the expected Kullback--Leibler divergence in the least favorable scenario to within asymptotically negligible terms as error probabilities go to zero. Moreover, based on simulation experiments, we can conclude that the  two tests  perform similarly  even for not too small error probabilities. 
Finally, we believe that the proposed approach can be extended to sequential testing of multiple hypotheses, a substantially more complex problem that we plan to consider elsewhere.   

%%%%_____________________________________
\section*{Acknowledgements}

We are grateful  to a referee whose comments improved the presentation. This work was supported by the U.S.\ Air Force Office of Scientific Research under MURI grant FA9550-10-1-0569, by the U.S.\ Defense Threat Reduction Agency under grant HDTRA1-10-1-0086, by the U.S.\  Defense Advanced Research Projects Agency under grant W911NF-12-1-0034 and by the U.S.\ National Science Foundation under grants CCF-0830419, EFRI-1025043, and DMS-1221888 at the University of Southern California, Department of Mathematics.

%%_____________________________________

%%%%%%%%%%%%%%%%%%%%%%%%%%%%%%%%%%%%%%%%%%%%%%%%%%%%%%%%%%%%%%%%%%%%%%%%%%

\vskip .65cm
\noindent
 University of Southern California, Department of Mathematics,
3620 S. Vermont Avenue, Los Angeles, CA 90089-2532, USA\\
%\vskip 2pt
\noindent E-mail: fellouri@usc.edu
\vskip 2pt

\

\noindent University of Southern California, Department of Mathematics,
3620 S. Vermont Avenue, Los Angeles, CA 90089-2532, USA\\
%\vskip 2pt
\noindent
E-mail: tartakov@usc.edu
\vskip .3cm

(Received March 2012; Revised January 2013; Accepted .....)

%%%%-----------------------------------------------------------------

\end{document}